\documentclass[font=11pt]{article}
\usepackage[margin=1in]{geometry}
\usepackage{mathtools, nccmath,mathrsfs, esint,dsfont}
\RequirePackage{amsthm,amsmath,amsfonts,amssymb, bbm,bm}
\RequirePackage[numbers]{natbib}
\RequirePackage[colorlinks,citecolor=blue,urlcolor=blue]{hyperref}
\RequirePackage{graphicx}
\usepackage[normalem]{ulem}
\numberwithin{equation}{section}
\theoremstyle{plain}
\newtheorem{theorem}{Theorem}[section]
\newtheorem{lemma}[theorem]{Lemma}
\newtheorem{prop}[theorem]{Proposition}
\newtheorem{corollary}[theorem]{Corollary}
\theoremstyle{remark}
\newtheorem{remark}[theorem]{Remark}
\theoremstyle{definition}
\newtheorem{defn}[theorem]{Definition}

\newtheorem{assumptionT}{Assumption}
\newenvironment{taggedassump}[1]
 {\renewcommand\theassumptionT{#1}\assumptionT}
 {\endassumptionT}
\newtheorem{example}[theorem]{Example}

\newcommand{\beq}{\begin{equation}}
\newcommand{\eeq}{\end{equation}}
\def\beqs#1\eeqs{
    \begin{equation}\begin{split}
    #1
    \end{split}\end{equation}
}
\def\beqsj #1\eeqsj
{
    \begin{equation}
    \setlength{\jot}{10pt}
    \begin{split}
    #1
    \end{split}\end{equation}
}

\def\beqsn#1\eeqsn{
    \begin{equation*}\begin{split}
    #1
    \end{split}\end{equation*}
}
\DeclareMathOperator*{\argmin}{argmin}
\DeclareMathOperator*{\argmax}{argmax}
\newcommand{\Var}{{\mathrm{Var}}}

\newcommand{\bvmgam}{\gamma^*}
\newcommand{\f}{v}
\newcommand{\tr}{\intercal}
\newcommand{\tmin}{{\theta_{\mathrm{min}}^*}}

\newcommand{\iid}{\stackrel{\mathrm{i.i.d.}}{\sim}}

\newcommand{\prior}{\pi_{0}} 

\newcommand{\nlp}{v_{0}} 
\newcommand{\nll}{\ell}

\newcommand{\les}{\lesssim}
\newcommand{\nllinfty}{\ell^*}
\newcommand{\MLE}{\hat\theta}

\newcommand{\one}{\mathds{1}}
\newcommand{\MAP}{\tilde\theta}
\newcommand{\UU}{{\mathcal U}}
\newcommand{\ground}{{\theta^{*}}}
\newcommand{\bdelta}{\delta^*}
\newcommand{\bU}{\UU^*}
\renewcommand{\l}{\left}
\renewcommand{\r}{\right}

\newcommand{\PP}{\mathbb P}
\newcommand{\e}{\,\mathrm{exp}}

\newcommand{\la}{\langle}
\newcommand{\ra}{\rangle}
\newcommand{\lla}{\l\la}
\newcommand{\rra}{\r\ra}
\newcommand{\R}{\mathbb R}
\newcommand{\E}{\mathbb E\,}
\newcommand{\TV}{\mathrm{TV}}
\newcommand{\KL}[2]{\mathrm{KL}\l(\,#1\;||\;#2\r)}
\newcommand{\F}{{F^*}}
\newcommand{\FI}[2]{\mathrm{FI}\l(\,#1\;||\;#2\r)}

\newcommand{\Cpri}{{M_0}}

\title{Improved dimension dependence in the Bernstein-von Mises Theorem via a new Laplace approximation bound}
\author{
 Anya Katsevich\thanks{This work was supported by NSF grant DMS-2202963}\\
  \texttt{akatsevi@mit.edu}
 }
\begin{document}
\maketitle

\abstract{
The Bernstein-von Mises theorem (BvM) gives conditions under which the posterior distribution of a parameter $\theta\in\Theta\subseteq\R^d$ based on $n$ independent samples is asymptotically normal. In the high-dimensional regime, a key question is to determine the growth rate of $d$ with $n$ required for the BvM to hold. We show that up to a model-dependent coefficient, $n\gg d^2$ suffices for the BvM to hold in two settings: arbitrary generalized linear models, which include exponential families as a special case, and multinomial data, in which the parameter of interest is an unknown probability mass functions on $d+1$ states.  Our results improve on the tightest previously known condition for posterior asymptotic normality, $n\gg d^3$. Our statements of the BvM are nonasymptotic, taking the form of explicit high-probability bounds. To prove the BvM, we derive a new simple and explicit bound on the total variation distance between a measure $\pi\propto e^{-nf}$ on $\Theta\subseteq\R^d$ and its Laplace approximation. }

\maketitle

\section{Introduction}
In frequentist inference, classical theory shows that the maximum likelihood estimator (MLE) $\hat\theta_n=\hat\theta_n(Y_{1:n})$ based on independent data $Y_{1:n}=\{Y_i\}_{i=1}^n$ is asymptotically normal in the large sample limit: \beq\label{freq}\mathrm{Law}(\hat\theta_n)\approx\mathcal N(\ground, I_n(\ground)^{-1}),\eeq where $\ground$ is the ground truth parameter and $I_n(\ground)=nI(\ground)$ is the scaled Fisher information matrix. The Bernstein-von Mises theorem (BvM) is the Bayesian counterpart of this result. Under conditions similar to the classical theory of the MLE and certain natural assumptions on the prior, it states that the posterior $\pi_n$ of the unknown parameter given independent data $Y_{1:n}$ is also asymptotically normal,
\beq\label{bvmgam-def-intro}\pi_n(\cdot\mid Y_{1:n})\approx\bvmgam_n:=\mathcal N(\hat\theta_n, I_n(\ground)^{-1})\eeq  with high probability under the ground truth data generating process~\cite{vaart_1998}. Here, approximation accuracy is typically measured in total variation (TV) distance. The BvM is a frequentist view on Bayesian inference, because it breaks out of the classic Bayesian paradigm of \emph{fixed data} by making a statement about the posterior as a random object which varies with \emph{random draws} of the data. 
The BvM is an important result because it justifies Bayesian uncertainty quantification from the frequentist perspective. Namely, the BvM shows that Bayesian credible sets $B_n$ of level $1-\alpha$ are asymptotic frequentist confidence sets of the same level~\cite{kleijn2012bernstein}. This is true in the sense that $B_n$ can be expressed by taking a set $C_n$ of approximate probability $1-\alpha$ under the standard normal, scaling by the Fisher information, and recentering it about the MLE $\MLE_n$: $$B_n=\MLE_n+I_n(\ground)^{-1/2}C_n.$$

Thus in the limit, Bayesian and frequentist inference are two sides of the same coin. But interestingly enough, despite this, there is a notable difference between the conditions under which asymptotic normality has been shown to hold in Bayesian and frequentist inference. This difference occurs in the \emph{high-dimensional regime}, in which the parameter $\theta$ lies in $\Theta\subseteq\R^d$, where $d$ grows with sample size $n$. The question is, how fast can $d$ grow with $n$ while preserving asymptotic normality? 

In frequentist inference, the pioneering work~\cite{portnoy1988asymptotic} showed that for exponential families, the MLE is asymptotically normal provided $d^2/n\to0$. The work~\cite{he2000parameters} extended this result to more general parametric models, while~\cite{spokoiny2012parametric} proved finite sample error bounds on the Fisher expansion of the MLE (closely related to normality), provided $d^2\ll n$.

In a departure from the condition $d^2/n\to0$ on the frequentist side, all of the works on the BvM for parametric Bayesian inference in growing dimension required at least that $d^3/n\to0$~\cite{panov2015finite,lu2017bernstein}. The first works were by~\cite{ghosal1999asymptotic,ghosal2000}, proving BvMs for posteriors arising from linear regression models and exponential families, respectively.~\cite{boucheron2009discrete} proves a BvM for the posterior of a discrete probability mass function truncated to its first $d$ entries.~\cite{spokoiny2013bernstein} proves a BvM for parametric models under more general conditions, which include e.g. generalized linear models (GLMs).~\cite{belloni2014posterior} extends the BvM result of~\cite{ghosal2000} to include the case of curved exponential families.~\cite{lu2017bernstein} proves a BvM for nonlinear Bayesian inverse problems. See all of the above works for further references on BvMs with growing parameter dimension. 

This discrepancy between frequentist and Bayesian asymptotics motivates the question: is the more stringent requirement on the growth of $d$ with $n$ in the BvM an inherent difference between the limiting behavior of the posterior and the law of the MLE? In this work, we resolve this question in the negative.

\subsection{Model dependence in high-dimensional BvMs}\label{intro:model}
Before conveying our results, we first raise an additional important consideration in the study of high-dimensional BvMs. Clearly, the TV distance between $\pi_n$ and $\bvmgam_n$ depends not only on $d$ and $n$, but also on the statistical model for the data-generating process. And in the high-dimensional regime, the contribution to the TV error stemming from the model can itself depend on $d$. Thus the condition on the growth of $n$ relative to $d$ required for asymptotic normality is model-dependent.

To demonstrate this phenomenon, consider the aforementioned works~\cite{boucheron2009discrete} and~\cite{lu2017bernstein}. In~\cite{boucheron2009discrete}, one can read the proof of the main result Theorem 3.7 to find that the intermediate Proposition 3.10 implies the bound 
\beq\label{bvmpmf}\TV(\pi_n,\bvmgam_n)=\mathcal O(\sqrt{d^{3+\epsilon}/n\tmin})\eeq for any $\epsilon>0$. We have assumed a flat prior for simplicity, i.e. we set $w_n\equiv1$ in the statement of the proposition, and made some other simplifications. Here, $\tmin$ is the minimum probability in the unknown probability mass function over $d+1$ states, and therefore $\tmin\leq1/(d+1)$. Thus $n\geq d^{3+\epsilon}/\tmin\geq d^{4+\epsilon}$ is required. 

In the work~\cite{lu2017bernstein} on the BvM for inverse problems, it is shown that 
\beq\label{bvminv}\TV(\pi_n,\bvmgam_n)=\mathcal O\l(\sigma(d)^2\log\l(d\vee\sigma(d)\r)\sqrt{d^3/n}\r),\eeq
where $\sigma(d)$ characterizes the degree of ill-posedness of the inverse problem. The author states that $\sigma(d)$ grows as a power of $d$ for mildly ill-posed inverse problems, and exponentially in $d$ for severely ill-posed inverse problems.

To meaningfully compare BvM results across different works studying different statistical models, we separate the contribution to the TV error into a  ``model-dependent" factor and a ``universal" factor. We informally define the universal factor as the term in the error bound depending on $d$ and $n$ only. For example, $1/\sqrt{\tmin}$ and $\sigma(d)^2\log(d\vee\sigma(d))$ are the model-dependent terms in the above TV bounds of~\cite{boucheron2009discrete} and~\cite{lu2017bernstein}, respectively, while $\sqrt{d^{3+\epsilon}/n}$ and $\sqrt{d^3/n}$ are the universal factors, respectively. Thus to clarify our above overview of prior works, the stated condition $n\gg d^3$ stems from the universal factor appearing in the bounds. 

\subsection{Main Contributions}\label{intro:contribute}
We prove nonasymptotic statements of the BvM for several representative statistical settings, showing for the first time that, up to model-dependent factors, $n\gg d^2$ suffices for $\TV(\pi_n,\bvmgam_n)\to0$. Specifically, we prove this for posteriors $\pi_n$ over an unknown parameter $\theta$ in the following settings: 1) arbitrary generalized linear models (GLMs) with data $Y_i\sim p(\cdot\mid X_i^\tr\theta)$, $i=1,\dots,n$ for an exponential family $p$ in canonical form and 2) a multinomial observation $Y^n\sim\mathrm{Multinomial}(n, \theta)$, where $\theta$ is an unknown probability mass function (pmf) on $d+1$ states (as in~\cite{boucheron2009discrete}). In each case, we show that, up to negligible terms, 
\beq\label{bvm-bd}\TV(\pi_n,\bvmgam_n)\leq \bdelta_3\sqrt{d^2/n}\eeq with high probability. Here, $\bdelta_3$ is an explicit model dependent coefficient related to the third derivative of the population log likelihood. This quantity may certainly depend on $d$ in some cases, and the second setting is one such example, as discussed below.

Our first, GLM setting is very general, allowing for the exponential family $p(\cdot\mid\omega)$ to have multivariate parameter $\omega\in\R^k$, $k\geq1$. Thus, the $X_i$ are $d\times k$ matrices. The dimension $k$ is itself also allowed to grow with $n$. In particular, taking $k=d$ and $X_i=I_d$ for all $i=1,\dots,n$, we recover the setting of i.i.d. observations of a $d$-parameter exponential family $p$.  Our BvM in the GLM setting can be compared to the works~\cite{ghosal2000,belloni2014posterior} on the BvM for exponential families and~\cite{spokoiny2013bernstein} on the BvM for univariate ($k=1$) GLMs. To make this comparison, some work is required to first recover explicit TV bounds from the proofs of~\cite{ghosal2000,belloni2014posterior}. With some effort, the recovered bound in these two works, as well as the stated TV bound of~\cite{spokoiny2013bernstein}, can be brought into the structure discussed above (see Appendix~\ref{subsec:bell-spok}). In the end, we see that our BvM tightens the dimension dependence of the universal factor from $\sqrt{d^3/n}$ to $\sqrt{d^2/n}$, while our model-dependent factor $\bdelta_3$ is of the same order of magnitude as that of prior works. 

To apply our general result to a particular GLM of interest, it remains to bound the quantity $\bdelta_3$. We do so in two particular cases: an i.i.d. log-concave exponential family, and a logistic regression model with Gaussian design. In both cases we show that $\bdelta_3$ is bounded by an absolute constant, where in the second case, this holds with high probability over the random design. Thus $d^2\ll n$ suffices for the BvM to hold in these settings, with no extra dimension dependence stemming from the model. Moreover, our treatment of the random design case seems to be new; we are not aware of any works proving the BvM for a random design GLM.

In our second setting of inference of a pmf, the model-dependent factor is $\bdelta_3=1/\sqrt{\tmin}$, where $\tmin$ is as above, i.e. the minimum probability in the ground truth pmf. Thus $\bdelta_3\geq\sqrt d$ in this setting. Our BvM directly improves on the result of~\cite{boucheron2009discrete}: from $\TV(\pi_n,\bvmgam_n)=\mathcal O(\sqrt{d^3/n\tmin})$ to $\TV(\pi_n,\bvmgam_n)=\mathcal O(\sqrt{d^2/n\tmin})$.

The GLM setting differs from the pmf setting in that the former enjoys a stochastic linearity property, to use the terminology of~\cite{spokoiny2022SLS}, while the latter does not. Stochastic linearity means that the sample log likelihood deviates from the population log likelihood by a linear function of $\theta$, and this property simplifies certain elements of the BvM proof. Despite this difference between the two settings, we use the same overall structure for our two BvM proofs. 

The key element of this common proof structure is a bound on the TV accuracy of the Laplace approximation (LA) to $\pi_n$. The LA is another Gaussian approximation $\gamma_n$ to $\pi_n$; for more details, see Section~\ref{sub:intro:LA} below. We derive a new elegant and simple bound on the LA accuracy $\TV(\pi_n,\gamma_n)$ which is tight in its dimension dependence (tightness is known thanks to lower bounds on $\TV(\pi_n,\gamma_n)$ derived in~\cite{katskew}). This bound is at the heart of our BvM proof, and it may also be of independent interest due to its simplicity. 

\noindent To summarize, our main contributions are as follows.
\begin{itemize}
\item We prove the BvM for arbitrary GLMs under the condition $n\gg d^2$ up to a model-dependent contribution $\bdelta_3$, improving for the first time on the condition $n\gg d^3$ from prior work.
\item We study two special cases of a GLM: log concave exponential families and logistic regression with Gaussian design, showing $\bdelta_3$ only contributes a constant factor in both cases.
\item We are the first to prove the BvM with high probability over the distribution of the random design in a GLM, in the high-dimensional regime.
\item We prove the BvM for a posterior over an unknown pmf, again showing $n\gg d^2$ suffices up to the explicit model-dependent contribution $\bdelta_3=1/\sqrt{\tmin}$.
\item All of our statements of the BvM are nonasymptotic, explicitly quantifying both the bound on $\TV(\pi_n,\bvmgam_n)$ and the probability with which the bound holds. Only one other work~\cite{spokoiny2013bernstein} has done this (with the worse rate $\sqrt{d^3/n}$).
\item Our proof is based on a new and simple bound on the TV accuracy of the Laplace approximation to a posterior.
\end{itemize}

\subsection{Proof via Laplace approximation}\label{sub:intro:LA} The main tool in our proof is a new bound on the accuracy of the Laplace approximation (LA). The LA is another Gaussian distribution similar in spirit to $\bvmgam_n$, but which can be computed in practice from given data. Namely, the LA to $\pi_n$ is given by $$\gamma_{n}:=\mathcal N(\MAP_n,\,-\nabla^2(\log\pi_n)(\MAP_n)^{-1}),\qquad \MAP_n:=\argmax_{\theta\in\R^d}\pi_n(\theta).$$ Bayesian inference can be very costly in high dimensions, since typical quantities of interest such as the posterior mean, covariance, and credible sets involve taking integrals against or sampling from $\pi_n$. Making the approximation $\pi_n\approx\gamma_n$ can simplify these computations. Here, we prove a new bound on the TV accuracy of the LA, which essentially takes the form 
\beq\label{laplace-bound}\TV(\pi_n,\gamma_n)\leq \delta_3\sqrt{d^2/n}.\eeq See Theorem~\ref{corr:main} for the precise statement. Here, $\delta_3$ depends on the third derivative of the function $\log\pi_n$ in a neighborhood of $\tilde\theta_n$; it is intimately related to  the ``model-dependent" factor $\bdelta_3$ in our bound~\eqref{bvm-bd} in the statement of the BvM.

The bound~\eqref{laplace-bound} on $\TV(\pi_n,\gamma_n)$ is the key ingredient in our BvM proof. Namely, we apply the triangle inequality $\TV(\pi_n,\bvmgam_n)\leq \TV(\pi_n,\gamma_n) + \TV(\gamma_n,\bvmgam_n)$, and use the bound~\eqref{laplace-bound} on the first term on the righthand side. Since $\delta_3$ is random (depending on the log posterior, which in turn depends on the data $Y_{1:n}$), it remains to prove a high-probability, deterministic upper bound on $\delta_3$. The second term $\TV(\gamma_n,\bvmgam_n)$ is a TV distance between two Gaussians, and can be bounded in a straightforward manner. There is an additional first step in which we quantify the effect of removing the prior from $\pi_n$, but we neglect this technicality in the present overview to simplify the discussion. 

The LA bound~\eqref{laplace-bound} has two important features which are noteworthy on their own and which have bearing on the BvM result. First, we see that the ``universal" factor in the upper bound is $\sqrt{d^2/ n}$. This directly translates to the tightened dimension dependence in the BvM. Second, the bound is remarkably simple and transparent, which allows for a very clear-cut proof of the BvM.  

Let us put our bound~\eqref{laplace-bound} into context. Like prior work on the BvM, most prior work on the LA accuracy in high dimensions also required $d^3\ll n$. In particular, the three works~\cite{dehaene2019deterministic,spokoiny2023dimension,helin2022non} all show, under slightly varying conditions, a bound of the form $\TV(\pi_n,\gamma_n)\leq  \delta_3'\sqrt{d^3/n}$. Here, $\delta_3'$ are various model-dependent coefficients analogous to $\delta_3$ from~\eqref{laplace-bound}. Recently, however, the works~\cite{bp,katskew} obtained bounds on the TV distance with ``universal" factor $d/\sqrt n$. The bound we derive here is significantly simpler than that of~\cite{bp,katskew}; see Section~\ref{subsec:discuss} for a more detailed comparison.

\subsection{Other related works.}\label{related}We note that in certain special cases, the limit of posterior distributions can be obtained in regimes in which $d^2/n$ is not small, including infinite-dimensional regimes. However, the \emph{type} of limit, and the \emph{shape} of the limiting posterior distribution, may be different. 

The work~\cite{bontemps2011bvm} proves the BvM for Gaussian linear regression in increasing dimensions. This is a case in which the log likelihood is quadratic, so that the only non-Gaussian contribution to the posterior stems from the prior. The author shows that under certain conditions on the prior, $d\ll n$ suffices for the BvM to hold. This weaker requirement on $n$ (compared to our $n\gg d^2$) is consistent with our findings. Indeed, our analysis shows that the main source of error in the BvM vanishes when the log likelihood is quadratic; see Remark~\ref{rk:quad}.

The work~\cite{castillo2015bayesian} proves a BvM for the coefficient vector in a sparse linear regression model, in which $d$ may be larger than $n$ but the true coefficient vector has few nonzero entries (exactly how many nonzero entries are allowed depends on various problem parameters). The authors show that the posterior is asymptotically a \emph{mixture} of degenerate normal distributions (owing in part to a special prior construction incorporating sparsity). In particular, the posterior marginal of each coefficient is a mixture of a point mass at zero and a continuous distribution. This has implications for the construction of credible sets.

In infinite dimensions,~\cite{freedman1999wald} gives an example in which the posterior distribution of a functional of an infinite-dimensional parameter converges in distribution to a Gaussian, but which gives rise to credible intervals differing from frequentist confidence intervals. The work~\cite{castillo2013nonparametric} proves a BvM for the Gaussian white noise model, in the sense that the authors show the posterior converges to the ``correct" (from the frequentist perspective) normal distribution in a \emph{weaker} sense than total variation norm. The notion of convergence is still strong enough to prove linear functionals of the parameter are asymptotically normal, centered at efficient estimators and with asymptotically efficient variance. Moreover, certain weighted $L^2$-credible ellipsoids coincide asymptotically with their frequentist counterparts.


Since the BvM is closely related to a high probability bound on the LA error, we mention a few works which prove the latter. This is in contrast to the Laplace bounds cited above, which are for a fixed realization of the data.
In~\cite{barber2016laplace}, the authors prove that for sparse generalized linear models with $n$ samples and $q$ active covariates, the relative LA error of the posterior normalizing constant is bounded by $(q^3\log^3n/n)^{1/2}$ with high probability. In~\cite{tang2021laplace}, the authors study the pointwise ratio between the density of the posterior $\pi_n$ and that of the LA $\gamma_n$. The authors show that for generalized linear models with Gaussian design, the relative error is bounded by $d^3\log n/n$ with probability tending to 1, provided $d^{3+\epsilon}\ll n$ for some $\epsilon>0$. For logistic regression with Gaussian design, they obtain the improved rate $d^2\log n/n$ provided $d^{2.5}\ll n$. 
Finally,~\cite{dehaene2019deterministic} shows $\KL{\pi_n}{\gamma_n}$ goes to 0 with probability tending to 1 if $d^3/n\to0$. 

Finally, we note that the LA is just one approach to approximate a posterior measure by a simpler distribution. Another approach that is similar in spirit is to approximate the posterior $\pi_n$ with a log concave distribution~\cite{nickl2022polynomial,bohr2024logconcave}. Roughly speaking, the log concave approximation coincides with $\pi_n$ in a neighborhood of the mode $\tilde\theta_n$ (where the negative log posterior is indeed convex), and it ``convexifies" the tails of $-\log\pi_n$. The two works~\cite{nickl2022polynomial,bohr2024logconcave} prove bounds on the accuracy of such log concave approximations in the context of high-dimensional Bayesian nonlinear inverse problems. Under suitable restrictions on the growth of $d$ relative to $n$, it is shown that the Wasserstein-2 distance between the two measures is exponentially small in $n$ with high probability under the ground truth data generating process.


\subsection*{Organization} 
The rest of the paper is organized as follows. In Section~\ref{gensec} we state, discuss, and outline the proof of our new Laplace approximation bound in a general setting, where $\pi_\f\propto e^{-n\f}$ need not have the structure of a posterior. In Section~\ref{sec:bayes} we introduce the Bayesian setting and use our Laplace bound to prove some key preliminary results for the BvM proofs. In Sections~\ref{sec:log} and~\ref{sec:pmf} we prove the BvM for posteriors stemming from GLMs and observations of a discrete probability distribution, respectively. Omitted proofs can be found in the Appendix.

\subsection{Notation}\label{notation} The letter $C$ always denotes an absolute constant. The inequality $a\lesssim b$ denotes $a\leq Cb$ for an absolute constant $C$. The statement that $a\les b$ on a random event indicates that there is an absolute nonrandom constant $C$ such that $a\leq Cb$ on the event. We let $|\cdot|$ denote the standard Euclidean norm in $\R^d$. If $A\in\R^{d\times d}$ is a symmetric positive definite matrix and $u\in\R^d$ is a vector, we let $$|u|_A=|A^{1/2}u|.$$ Let $k\geq 1$ and $S$ be a symmetric $k$-linear form, identified with its symmetric tensor representation. We define the $A$-weighted operator norm of $S$ to be
\beq\label{TA}\|S\|_A:=\sup_{|u|_A=1}\la S, u^{\otimes k}\ra = \sup_{|u|_A=1}\sum_{i_1,\dots,i_k=1}^dS_{i_1i_2\dots i_k}u_{i_1}u_{i_2}\dots u_{i_k}.\eeq  
For example, if $k=2$ then $S$ is a symmetric matrix and $\|S\|_A=\|A^{-1/2}SA^{-1/2}\|$, where $\|\cdot\|$ is the standard matrix operator norm. Note that if $k=1$ and $S$ is a 1-linear form identified with its vector representation, then 
$$\|S\|_A=|A^{-1/2}S|,\qquad |S|_A = |A^{1/2}S|.$$ Thus it is important to distinguish between the meaning of $\|\cdot\|_A$ and $|\cdot|_A$ when the quantity inside the norm is a vector. By Theorem 2.1 of~\cite{symmtens}, for symmetric tensors the definition~\eqref{TA} coincides with the standard definition of operator norm:
$$\sup_{|u_1|_A=\dots=|u_k|_A=1} \la S, u_1\otimes\dots\otimes u_k\ra = \|S\|_A=\sup_{|u|_A=1}\la S, u^{\otimes k}\ra.$$ 

\section{Laplace approximation bound}\label{gensec}
We begin by stating and discussing our general bound on the Laplace approximation accuracy. In Section~\ref{subsec:outline}, we outline the proof. In Section~\ref{subsec:discuss}, we compare our bound to other bounds in the literature which achieve the tight $d/\sqrt n$ rate. We consider a general setting which need not stem from Bayesian inference. Throughout the paper, we let $\Theta$ denote a convex subset of $\R^d$. The value $n$ should be thought of as some large positive number. 
\renewcommand{\f}{f}
\begin{defn}\label{def:meas}
Let $\f\in C^2(\Theta)$. Whenever $e^{-n\f}$ is integrable, we define the probability density $\pi_{\f}(\theta)\propto e^{-n\f(\theta)}$, $\theta\in\Theta$. 
If $\f$ has a unique strict global minimizer $\MLE\in\Theta$, we define the following \emph{Laplace approximation} (LA) to $\pi_\f$:
$$\gamma_\f=\mathcal N(\MLE, (nH)^{-1}),\qquad H:=\nabla^2\f(\MLE).$$
\end{defn}
\begin{theorem}\label{corr:main}
Let $\f\in C^2(\Theta)$ be convex with unique strict global minimizer $\MLE$. Assume there is $r\geq6$ such that $\UU(r)\subset\Theta$, where $\UU(r)=\{ |\theta-\MLE|_H\leq r\sqrt{d/n}\}$. Let 
\beq\label{del3def}\delta_3(r)=\sup_{\theta\in\UU(r),\theta\neq\hat\theta}\frac{\|\nabla^2\f(\theta)-\nabla^2\f(\MLE)\|_H}{|\theta-\MLE|_H}\eeq 
 and suppose $r\delta_3(r)\sqrt{d/n}\leq1/2$. Then $e^{-n\f}\in L^1(\Theta)$ so $\pi_\f\propto e^{-n\f}$ is well-defined, and we have
\beq\label{mainbound}
\TV(\pi_\f,\;\gamma_\f)\leq \frac{\delta_3(r)d}{\sqrt {2n}}+ 3e^{-dr^2/9}.
\eeq 
 \end{theorem} Recall from Section~\ref{notation} that $\|\nabla^2\f(\theta)-\nabla^2\f(\MLE)\|_H=\|H^{-1/2}(\nabla^2\f(\theta)-\nabla^2\f(\MLE))H^{-1/2}\|$ and $|\theta-\MLE|_H=|H^{1/2}(\theta-\MLE)|$, where $\|\cdot\|$ is the standard matrix operator norm and $|\cdot|$ is the Euclidean norm on $\R^d$. We hope that the remarkable simplicity of the bound~\eqref{mainbound} will lead to its broad applicability. See the discussion in Section~\ref{subsec:discuss} on how our bound compares to previous bounds in the literature.
  \begin{remark}[Convexity assumption]\label{rk:cvx}
To prove~\eqref{mainbound}, we use convexity only through the linear growth lower bound~\eqref{nllgrowth} in Lemma~\ref{lingrow} below. Therefore, convexity of $\f$ can be replaced by~\eqref{nllgrowth} in the statement of Theorem~\ref{corr:main}. However, under convexity~\eqref{nllgrowth} holds for any $r\geq0$, giving us the flexibility to optimize the bound~\eqref{mainbound} over $r$.
\end{remark}
\begin{remark}[Smoothness and $\delta_3$]\label{rk:smooth}As noted in the introduction, the only smoothness we require is that $\f\in C^2(\Theta)$ and $\|\nabla^2\f(\theta)-\nabla^2\f(\MLE)\|_H=\mathcal O(|\theta-\MLE|_H)$ for $\theta\in\UU(r)$, to ensure that $\delta_3(r)<\infty$. Thus for example $\f(\theta)=|\theta|^2/2+|\theta|^3/6$ satisfies the conditions of the theorem: it is convex, strictly minimized at $\MLE=0$, and $\f\in C^2$, with $$\nabla^2\f(\theta)=\l(1+\frac12|\theta|\r)I_d + \frac{1}{2|\theta|}\theta\theta^\tr, \quad\theta\neq0$$ and $H=\nabla^2\f(0)=I_d$. It is straightforward to show that
$$\|\nabla^2f(\theta)-\nabla^2f(0)\| =\l \|\frac12|\theta|I_d + \frac{1}{2|\theta|}\theta\theta^\tr\r\| = |\theta|$$ for all $\theta\neq0$, and therefore $\delta_3(r)=1$ for all $r>0$, even though $\nabla^3\f$ does not exist at $\theta=0$. If $\f$ \emph{is} $C^3$ in a neighborhood of $\MLE$, then one can bound $\delta_3(r)$ via
\beq\label{del3redef}\delta_3(r) \leq\sup_{\theta\in\UU(r)}\|\nabla^3\f(\theta)\|_H.\eeq 
\end{remark}
 \begin{remark}[Dependence of $\delta_3$ on $d$]\label{rk:del3dim}
 We have not indicated the dependence of $\delta_3(r)$ on $d$ or $n$, but it does in fact depend on both parameters. The dependence of $\delta_3(r)$ on $n$ is weak, but \emph{the dependence on $d$ may be significant}, simply by virtue of the fact that $f$ is defined on (a subset of) $\R^d$. For example, consider the function $f(\theta)=|\theta|^2/2+| 1^\tr\theta|^3/6$, where $1$ is the all ones vector in $\R^d$. The minimizer of $f$ is $\hat\theta=0$,with $H=\nabla^2f(0)=I_d$. Using that $\nabla^2f(\theta)=I_d+| 1^\tr\theta| 1 1^\tr$, it is straightforward to compute that $\delta_3(r)=d^{3/2}$ for all $r>0$.
 \end{remark}
\begin{remark}[Choice of $r$]\label{rk:chooser} Note that~\eqref{mainbound} is valid for all $r\geq6$ such that $\UU(r)\subseteq\Theta$ and $r\delta_3(r)\sqrt{d/n}\leq1/2$. Thus $r$ can be chosen to optimize the bound. If $d$ is very large, then the second term in the bound~\eqref{mainbound} is exponentially small in $d$, so it suffices to choose $r=6$. If $d$ cannot be considered too large, then the optimal choice of $r$ depends on how $\delta_3(r)$ scales with $r$ and $d$. The following example shows how to choose $r$ in the special case when $\delta_3(r)$ is uniformly bounded.\end{remark}
\begin{example}[Uniformly bounded third derivative]\label{ex:3}Suppose $\f\in C^3(\Theta)$ and $\sup_{\theta\in\Theta}\|\nabla^3\f(\theta)\|_H\leq M$, where $M$ may depend on $d$. To give one example in which the third derivative is indeed uniformly bounded over $\Theta$, see the logistic regression setting in Section~\ref{subsec:logreg}, and the bound~\eqref{bdel3-log} in particular (the quantity $\bdelta_3$ in that bound is closely related to the above $\delta_3$).

Under the stated assumptions, we claim that
\beq\label{mainboundC3}
\TV(\pi_\f,\;\gamma_\f) \leq M\l(\frac{d}{\sqrt {2n}}+ \frac{9}{\sqrt n}\r)
\eeq for all $n/d\geq (12M)^2$. Indeed, note that $M$ is a uniform bound on $\delta_3(r)$ over all $r\geq0$, so~\eqref{mainbound} is optimized by taking $r$ as large as possible while satisfying the constraint $r\delta_3(r)\sqrt{d/n}\leq 1/2$. Thus we take $r=\sqrt{n/d}/2M$, which satisfies $r\geq6$ provided $n/d\geq (12M)^2$. Substituting this choice of $r$ into the second term of~\eqref{mainbound}, and using $\delta_3(r)\leq M$ in the first term gives $$\TV(\pi_\f,\;\gamma_\f)\leq \frac{Md}{\sqrt{2n}}+ 3\e(-n/(36M^2)).$$ Applying the inequality $e^{-x^2/4}\leq 1/x$ with $x=\sqrt n/(3M)$ yields~\eqref{mainboundC3}.\end{example}
\begin{remark}[Affine invariance]\label{rk:aff}
If $T:\R^d\to\R^d$ is a bijective affine map, then the pushforward $T_{\#}\pi_\f$ of $\pi_\f$ under $T$ is given by $T_{\#}\pi_\f = \pi_{\f\circ T^{-1}}$. Also, one can check that $\gamma_{\f\circ T^{-1}}=T_{\#}\gamma_\f$. Since the TV distance is invariant under bijective coordinate transformations, we therefore have $\TV(\pi_\f,\gamma_\f) = \TV(T_{\#}\pi_\f, T_{\#}\gamma_\f) = \TV(\pi_{\f\circ T^{-1}}, \gamma_{\f\circ T^{-1}}).$ Due to this invariance, a good upper bound on $\TV(\pi_\f,\gamma_\f)$ should also be affine invariant, and so should the conditions under which the bound holds. It is straightforward to check that Theorem~\ref{corr:main} has this property.
\end{remark}
\begin{remark}[Application to Bayesian inference]In Bayesian inference, the posterior distribution can be written as $\pi_{\f_n}\propto e^{-n\f_n}$ for some function $\f_n$ which is weakly dependent on $n$. The standard quantity of interest in Bayesian inference are posterior credible sets of level $1-\alpha$; that is, central regions $A$ such that $\pi_{\f_n}(A)=1-\alpha$~\cite{bda3}. The LA $\gamma_{\f_n}$ to $\pi_{\f_n}$ immediately yields explicit, approximate credible sets $\hat A$ based on level sets of $\gamma_{\f_n}$, and Theorem~\ref{corr:main} can be used to bound the deviation of the actual probability $\pi_{\f_n}(\hat A)$ from $1-\alpha$. 
 \end{remark}

\subsection{Proof}\label{subsec:outline}
The assumptions of Theorem~\ref{corr:main} imply two key properties used in the proof: first, $\pi_\f$ is strongly log concave in the neighborhood $\UU(r)$ of $\MLE$. Second, $\f$ grows linearly away from $\MLE$ in $\UU(r)^c$. These two observations are stated in the following lemma:

\begin{lemma}[Hessian lower bound and linear growth]\label{lingrow}Under the assumptions of Theorem~\ref{corr:main}, it holds
\beq\label{nll-lowerbd}
\nabla^2\f(\theta)\succeq \frac{1}{2}H,\quad\forall \theta\in\UU(r).
\eeq
Furthermore, we have the following growth bound:
\beq\label{nllgrowth}
\f(\theta)-\f(\MLE)\geq\frac r4\sqrt{d/n}|\theta-\MLE|_H\qquad \forall \theta\in\Theta\setminus\UU(r).
\eeq 
\end{lemma}See Appendix~\ref{app:gensec} for the proof of the lemma. The bound~\eqref{nll-lowerbd} does not require convexity. Similarly,~\eqref{nllgrowth} only uses convexity via the property of convex functions that the infimum of $(\f(\theta)-\f(\MLE))/|\theta-\MLE|_H$ over $\theta\in\UU(r)^c$ is achieved on the boundary $\{|\theta-\MLE|_H=r\sqrt{d/n}\}$. Thus any other function with this property will satisfy the lemma. Note also that~\eqref{nllgrowth} proves $e^{-n\f}\in L_1(\Theta)$. 

Next, the following lemma presents the main tools in the proof of Theorem~\ref{corr:main}.
\newcommand{\g}{\hat\f}
\newcommand{\nnu}{\hat\mu}
\begin{lemma}\label{lma:TVprelim}
Let $\mu$ and $\nnu$ be two probability densities on $\R^d$ and let $\UU\subset\R^d$ be such that $\mu(\UU)>0,\nnu(\UU)>0$. Let $\mu\vert_\UU$, $\nnu\vert_\UU$ be the restrictions of $\mu$ and $\nnu$ to $\UU$, respectively, renormalized to be probability measures. Then
\beq\label{eq:TVprelim}
\TV(\mu,\nnu)\leq\mu(\UU^c) +\nnu(\UU^c) +\TV\l(\mu\vert_\UU, \nnu\vert_\UU\r).
\eeq
Furthermore, suppose $\mu$ and $\nnu$ are strictly positive on $\UU$ with $\mu\propto e^{-nf}$ and $\nnu\propto e^{-n\g}$ where $\f\in C^2(\UU)$, $\g\in C^1(\UU)$. Suppose also that $\nabla^2\f(\theta)\succeq \lambda H$ for all $\theta\in \UU$ and some $H\succ0$. Then
\beq\label{eq:FI}\TV\l(\mu\vert_\UU, \nnu\vert_\UU\r)^2\leq \frac{n}{4\lambda}\E_{X\sim\nnu\vert_\UU}\l[\|\nabla (\f-\g)(X)\|_H^2\r],\eeq and therefore the total TV distance between $\mu$ and $\nnu$ is bounded as
\beq\label{TVFI}
\TV(\mu,\nnu)\leq \mu(\UU^c) +\nnu(\UU^c)  + \sqrt{\frac n{4\lambda}}\E_{X\sim\nnu\vert_\UU}\l[\|\nabla (\f-\g)(X)\|_H^2\r]^{\frac12}.\eeq
\end{lemma}See Appendix~\ref{app:lma:TVprelim} for the proof of this lemma. The bound~\eqref{eq:FI} stems from a log-Sobolev inequality (LSI)~\cite{bakry2014analysis}, which can be applied since $\mu$ is strongly log-concave in the neighborhood $\UU$. The inspiration for this technique comes from~\cite{bp}, who also prove a bound on $\TV(\pi_\f,\gamma_\f)$ in the context of Bayesian inference. See Section~\ref{subsec:discuss} for a comparison between our bound and that of~\cite{bp}.

We apply~\eqref{TVFI} with $\UU=\UU(r)$, $\mu=\pi_\f$ (extended to be zero in $\R^d\setminus\Theta$) and $\nnu=\gamma_\f\propto e^{-n\g}$, where $\g(\theta)=\frac{1}{2}|\theta - \MLE|^2_H$. Since $\nabla^2\f\succeq\frac{1}{2}H$ on $\UU(r)$ by~\eqref{nll-lowerbd}, the function $f$ enjoys \emph{local} strong convexity. In particular, the lower bound condition in the lemma is satisfied with $\lambda=1/2$. Therefore, in our setting, the bound~\eqref{TVFI} takes the form
\beq\label{TVFI-nll}
\TV(\pi_\f,\gamma_\f)\leq \pi_\f(\UU^c) + \gamma_\f(\UU^c) + \sqrt{\frac n2}\E_{\theta\sim\gamma_\f\vert\UU}\l[\|\nabla \f(\theta)-H(\theta-\MLE)\|^2_H\r]^{\frac12}.
\eeq We see that to bound the TV distance between $\pi_\f$ and $\gamma_\f$, it suffices to bound the two tail integrals and the local expectation in~\eqref{TVFI-nll}. The tail integral $\pi_\f(\UU^c) = \int_{\UU^c}e^{-n\f}/\int e^{-n\f}$ can be bounded using the linear growth of $\f$ in $\UU^c$. The bound on $\gamma_\f(\UU^c)$ follows by standard Gaussian tail inequalities. Finally, to bound the local expectation we use a Taylor expansion of $\nabla \f(\theta)$ about $\MLE$.to get
$$\nabla \f(\theta)-H(\theta-\MLE)=\nabla f(\theta)-\nabla f(\MLE)-\nabla^2f(\MLE)(\theta-\MLE)= (\nabla^2\f(\xi)-\nabla^2\f(\MLE))(\theta-\MLE)$$ for a point $\xi$ between $\theta$ and $\MLE$. From here we use the definition of $\delta_3$ and simple Gaussian calculations to show that the third term in~\eqref{TVFI-nll} is bounded as $\delta_3(r)d/\sqrt n$.
See Appendix~\ref{app:prop:main} for these calculations, which finish the proof of Proposition~\ref{corr:main}.

\subsection{Comparison with~\cite{bp} and~\cite{katskew}}\label{subsec:discuss}

So far, only two other works have obtained bounds on $\TV(\pi_\f,\gamma_\f)$ scaling as $d/\sqrt n$ --- that of~\cite{bp} and~\cite{katskew}. We compare~\eqref{mainbound} to these two bounds from the perspective of the BvM.

The bound of~\cite{bp} requires $f\in C^3(\Theta)$. Here, we require that $f\in C^2(\Theta)$ and that $\delta_3(r)$ defined in~\eqref{del3def} is finite. The bound of~\cite{bp} is also significantly more complex than ours; see Theorem 3.1 in~\cite{bp} (which relies on the quantities defined in Section 2) compared to our~\eqref{mainbound}. The bound of~\cite{katskew} involves fourth derivatives of $f$. Thus greater regularity of $f$ is required, and the presence of fourth derivatives makes the bound more complex than~\eqref{mainbound}.

The weaker regularity condition required in our bound~\eqref{mainbound} translates into weaker regularity of the log likelihood in the statement of the BvM. As can be seen from the definition of event $E_3$ in~\eqref{events} and the third condition in~\eqref{bUcond}, only boundedness of the Lipschitz constant of the log likelihood's second derivative is needed.  Also, as already mentioned, the simplicity of~\eqref{mainbound} is useful to obtain a streamlined proof and statement of the BvM. 

Finally, let us compare our proof technique to that of~\cite{bp} and~\cite{katskew}. The key idea behind our proof is the same as in~\cite{bp}, that is, to use the log-Sobolev inequality in a local region. However, the implementation is different. Here, we (1) do not split up the function $f$ into a sum of two parts coming from the log likelihood and the log prior, (2) apply the log-Sobolev inequality in an affine invariant way, and (3) handle the tail integrals differently, relying on the linear growth of $f$ at infinity rather than the fact that the prior is a proper Lebesgue density integrating to 1.

The proof technique in~\cite{katskew} is quite different. In this work, the leading order term in the TV distance between $\pi_f$ and $\gamma_f$ is derived, and the main challenge is to bound the difference between $\TV(\pi_f,\gamma_f)$ and this leading term. This is done using high-dimensional Gaussian concentration. Overall, the calculation in~\cite{katskew} is more delicate, but leads to a more complicated bound involving fourth derivatives. 

\section{From Laplace approximation to Bernstein-von Mises: an overview}\label{sec:bayes}
In Section~\ref{subsec:bayes} we review the general Bayesian set-up and state the BvM result we wish to prove. In Section~\ref{sec:bvmprelim} we use Theorem~\ref{corr:main} to prove a key preliminary bound. Using this bound, it is then straightforward to finish the BvM proof for the specific statistical models considered in this work; we do so later in Sections~\ref{sec:log} and~\ref{sec:pmf}. Finally, in Section~\ref{sec:bvm:discuss} we discuss the conditions on the prior and how the condition $d^2\ll n$ arises in the BvM.

\subsection{Bayesian preliminaries and statement of BvM}\label{subsec:bayes}
Let $\{P^n_\theta\; : \; \theta\in\Theta\}$ be a parameterized family of probability distributions, with $\Theta$ an open convex subset of $\R^d$. We assume there is a fixed measure with respect to which $P^n_\theta$ has density $p^n(\cdot\mid\theta)$, for each $\theta\in\Theta$. We observe a sample $Y^n\sim P^n_\ground$ for a ground truth parameter $\ground\in\Theta$. We define the \emph{likelihood} $L:\Theta\to(0,\infty)$ by $L(\theta)=p^n(Y^n\mid\theta)$, and the \emph{negative normalized log likelihood} $\nll:\Theta\mapsto\R$ by 
$$\nll(\theta)= -\frac1n\log L(\theta)=-\frac1n\log p^n(Y^n\mid\theta).$$ The normalization by $1/n$ is natural in the standard case that $P^n_\theta$ is an $n$-fold product measure. In other words, $Y^n=(Y_1,\dots,Y_n)$ and $p^n(Y^n\mid\theta)=\prod_{i=1}^np_i(Y_i\mid\theta)$. We then have that $\log L$ is a sum of $n$ terms. When it exists and is unique, we define 
$$\MLE = \argmin_{\theta\in\Theta}\nll(\theta),$$ the maximum likelihood estimator (MLE). We define the \emph{negative population log likelihood} to be the function
$$\nllinfty(\theta)=\E_{Y^n\sim P^n_\ground}[\nll(\theta)]=-\frac1n\E_{Y^n\sim P^n_\ground}\l[\log p^n(Y^n\mid\theta)\r].$$
\begin{taggedassump}{A1}\label{A1}
The random function $\nll$ is convex and belongs to $C^2(\Theta)$ with probability 1 with respect to the distribution $Y^n\sim p^n(\cdot\mid\ground)$. The function $\nllinfty$ belongs to $C^2(\Theta)$, and the Fisher information matrix 
$$
\F: = \nabla^2\nllinfty(\ground)
$$ is strictly positive definite.
\end{taggedassump}

Let the prior probability density over $\theta$ be denoted $\prior$. We assume without loss of generality that $\Theta$ lies in the support of $\prior$; otherwise, redefine $\Theta$ to be its intersection with the support of $\prior$. Thus there is a function $\nlp:\Theta\to\R$ such that $$\prior(\theta)\propto e^{-\nlp(\theta)},\quad\theta\in\Theta.$$ Finally, the \emph{posterior} distribution is the probability density $\pi_v$ such that
$$\pi_v(\theta)\propto e^{-n\nll(\theta)}\prior(\theta)\propto e^{-nv(\theta)}, \qquad v= \nll+n^{-1}\nlp,\quad\theta\in\Theta.$$ \\

Classically, the BvM states that $\TV(\pi_v,\bvmgam)=o(1)$ with probability tending to 1 (under $P^n_\ground$) as $n\to\infty$~\cite{vaart_1998}. Here, 
\beq\label{def:bvmgam}\bvmgam =  \mathcal N(\hat\theta, (n\F)^{-1}).\eeq From the definition of $\bvmgam$, we see that a necessary ingredient to prove the BvM is showing the MLE exists and is unique with probability tending to 1.

Here, we will be interested in a nonasymptotic version of the BvM. Namely, we will prove a result of the form: ``under certain conditions on $d$, $n$, and the statistical model, $\TV(\pi_v,\bvmgam)\leq \epsilon$ with probability at least $1-p$", for explicit quantities $\epsilon$ and $p$ depending on $d$, $n$, and the model. We are only aware of one other work which proves BvMs in this style~\cite{spokoiny2013bernstein}. Our nonasymptotic BvM results then easily lead to the more classical asymptotic statements.


\subsection{Proof outline and preliminary bounds}\label{sec:bvmprelim}
The proof idea is as follows. First, we show that discarding the prior has negligible effect by bounding $\TV(\pi_v,\pi_\nll)$. Here, $\pi_\nll\propto e^{-n\nll}$. Second, we bound $\TV(\pi_\nll,\gamma_\nll)$ by directly applying Theorem~\ref{corr:main} with $f=\nll$. Here, \beq\label{def:gammanll}\gamma_\nll=\mathcal N(\hat\theta, (n\nabla^2\nll(\hat\theta))^{-1})\eeq is the LA to $\pi_\nll$. Third, we bound $\TV(\gamma_\nll,\bvmgam)$ by comparing the inverse covariance matrices $n\nabla^2\nll(\MLE)$ and $n\nabla^2\nllinfty(\ground)$ of the two Gaussian distributions (which have the same mean). We then use the triangle inequality to bound $\TV(\pi_v,\bvmgam)$ via the sum of the three bounds. To be more precise, we show in this section that these three bounds hold on the event $E(s,\epsilon_2)$ from Definition~\ref{def:Ugen} below. 
To finish the BvM proof, it remains to prove this is a high-probability event. We will carry out this final probabilistic step using the particulars of each of our models in the following two sections. \\

We now introduce several quantities and specify the event $E(s,\epsilon_2)$.
\begin{defn}\label{def:Ugen}
Let
\begingroup
\beqs
\bU(s) &= \l\{\theta\in\R^d\; :\;|\theta-\ground|_\F\leq s\sqrt{d/n}\r\},\\
 \bdelta_{01}(s) &=\sup_{\theta\in\bU(s)}\|\nabla\log\prior(\theta)\|_\F,\\
 \Cpri^* &= d^{-1}\log\sup_{\theta\in\Theta}\prior(\theta)/\prior(\ground),
\eeqs 
\endgroup with the convention that $\bdelta_{01}(s)=\infty$ if the prior is not $C^1$ or not strictly positive on $\bU(s)$. Also, let $\bdelta_3:[0,\infty)\to[0,\infty)$ be some deterministic nondecreasing function to be specified. Then for $s,\epsilon_2\geq0$, define the events
\beqs\label{events}
E_1(s)& = \l\{\|\nabla\nll(\ground)\|_\F\leq  s\sqrt{d/n}\r\},\\
 E_2(\epsilon_2)&=\l\{\|\nabla^2\nll(\ground)-\F\|_\F \leq \epsilon_2\r\},\\
  E_3(s)&=\l\{ \sup_{\theta,\theta'\in\bU(s)}\frac{\|\nabla^2\nll(\theta)-\nabla^2\nll(\theta')\|_\F}{|\theta-\theta'|_\F} \leq\bdelta_3(s)\r\}.
\eeqs
Finally, let
\beqsn
E(s,\epsilon_2)&=E_1(s)\cap E_2(\epsilon_2)\cap E_3(2s),\\
\bar E(s,\epsilon_2)&=\{\exists!\,\mathrm{MLE}\,\MLE,\, |\MLE-\ground|_\F\leq s\sqrt{d/n}\}\cap E_2(\epsilon_2)\cap E_3(2s).
\eeqsn
\end{defn} Here, ``$\exists!\,\mathrm{MLE}\,\MLE$" is an abbreviation for ``the MLE exists and is unique". Also, recall from Section~\ref{notation} that $\|\nabla\nll(\ground)\|_\F=|\F^{-1/2}\nabla\nll(\ground)|$.

We will show that $E(s,\epsilon_2)\subset \bar E(s,\epsilon_2)$, so that in particular, a unique MLE $\MLE$ is guaranteed to exist on $E(s,\epsilon_2)$ and $\MLE$ is close to $\ground$. We then state our key bounds on the event $\bar E(s,\epsilon_2)$.
\begin{remark}\label{rk:reg}Consider the quantity $\delta_3$ defined in~\eqref{del3def}, when applied to $f=\nll$. There are several slight differences between $\delta_3$ and the above $\bdelta_3$. First, $\delta_3$ is a local Lipschitz constant of $\nabla^2\nll$ in the neighborhood $\UU(r)$, anchored at the point $\MLE$, whereas $\bdelta_3(s)$ is an upper bound on a ``uniform" Lipschitz constant in the neighborhood $\bU(s)$. Second, $\bdelta_3$ is a \emph{deterministic} function, bounding this \emph{random} Lipschitz constant uniformly over all realizations of $\nll$ in the event $E_3$. When applying Theorem~\ref{corr:main} to bound $\TV(\pi_\nll,\gamma_\nll)$, we will show that the random quantity $\delta_3$ appearing in the upper bound~\eqref{mainbound} can be further bounded by the deterministic $\bdelta_3$. 
\end{remark}
\begin{remark}[Intuition for events $E_1,E_2,E_3$]\label{interpreting-E}Let us show why we expect $E_1,E_2,E_3$ to be high-probability events in general (though as mentioned above, we will prove this rigorously using the particulars of our models). First, note that $\ground$ is the global minimizer of $\nllinfty$, since $\nllinfty(\theta)=\KL{P^n(\ground)}{P^n(\theta)}+\mathrm{const.}$. Thus $\nabla\nllinfty(\ground)=0$. Second, recall that $\F=\nabla^2\nllinfty(\ground)$. Thus the events $E_1,E_2$ are expressing that $\nabla^k\nll(\ground)\approx\nabla^k\nllinfty(\ground)$ for $k=1,2$, respectively. This is reasonable to expect when $n$ is large and $P^n_\theta$ is a product measure (with $Y^n=(Y_1,\dots,Y_n)$), since then $\nabla^k\nll(\ground)$ can be written as an average of the $n$ independent random variables $\nabla_\theta^k\log p_i(Y_i\mid\ground)$, $i=1,\dots,n$. To interpret the event $E_3$, consider the case when $\nll\in C^3(\Theta)$ with probability 1, and $P^n_\theta$ is still a product measure. Since $\nabla^3\nll\approx\nabla^3\nllinfty$ for large $n$, we expect $\sup_{\theta\in\bU(s)}\|\nabla^3(\nll-\nllinfty)(\theta)\|_\F\leq\epsilon_3(s)$ with high probability, for some suitably chosen $\epsilon_3(s)$. We can then take $\bdelta_3(s)=\epsilon_3(s)+\sup_{\theta\in\bU(s)}\|\nabla^3\nllinfty(\theta)\|_\F$.
\end{remark}
%
\begin{lemma}[Nonasymptotic BvM: preliminary lemma]\label{prop:lapMLE}Suppose $0\leq\epsilon_2\leq1/2$ and 
\beq\label{bUcond}s\geq12, \quad \bU(2s)\subset\Theta, \quad 2s\bdelta_3(2s)\sqrt{d/n}\leq 1/4.\eeq Then $E(s,\epsilon_2)\subseteq\bar E(s,\epsilon_2)$ and on $\bar E(s,\epsilon_2)$ it holds
\begin{align}
\TV(\pi_\nll, \gamma_\nll)&\les \bdelta_3(2s)\,\frac{d}{\sqrt n}+e^{-ds^2/36},\label{pretv1}\\
\TV(\gamma_\nll,\bvmgam) &\les s\bdelta_3(s)\frac{d}{\sqrt n} +\sqrt d\epsilon_2.\label{pretv3}
\end{align} 
If also $ \bdelta_{01}(2s)\leq\sqrt{nd}/6$ then we have the following inequality on $\bar E(s,\epsilon_2)$:
\beq
\TV(\pi_v, \pi_\nll)\les \frac{\bdelta_{01}(2s)}{\sqrt n} + e^{d[\Cpri^*  - (s/12)^2]}.\label{pretv2}
\eeq  In each of the above bounds, the suppressed constant is absolute, independent of all problem parameters. 
\end{lemma}
Let us briefly outline the proof. On $E_1(s)$ we have that $\|\nabla\nll(\ground)\|_{\F}$ is small. By the inverse function theorem, we can then show that there is a point $\MLE$ near $\ground$ such that $\nabla\nll(\MLE)=0$. Specifically, we show that this point satisfies $|\MLE-\ground|_\F \leq s\sqrt{d/n}$. Next we can combine the conditions from $E_2(\epsilon_2)$ and $E_3(s)$ to show $\nabla^2\nll(\MLE)\succ0$. Thus since $\nll$ is convex, we conclude that $\MLE$ is the unique global minimizer of $\nll$, i.e. it is the MLE. To prove the bound~\eqref{pretv1}, we use the definition of $E_2(\epsilon_2)$ and $E_3(s)$, and the proximity of $\MLE$ to $\ground$, to bound the quantity $\delta_3$ from Theorem~\ref{corr:main} via the quantity $\bdelta_3$. We then directly apply Theorem~\ref{corr:main}. The proof of~\eqref{pretv2} is similar to the proof of Theorem~\ref{corr:main}, with Lemma~\ref{lma:TVprelim} (i.e. the log-Sobolev inequality) as the key tool. Finally, the bound~\eqref{pretv3} is straightforward since $\gamma_\nll$ and $\bvmgam$ are Gaussians with the same mean, and their inverse covariances are close on the event $E_2(\epsilon_2)$. (Recall the definition of $\gamma_\nll$ and $\bvmgam$ from~\eqref{def:gammanll} and~\eqref{def:bvmgam}.) See Appendix~\ref{app:towardsbvm} for the full proof.\\

Given Lemma~\ref{prop:lapMLE}, we can finish the nonasymptotic BvM proof by finding $s,\epsilon_2,\bdelta_3$ which balance the two requirements that $E(s,\epsilon_2)$ is a high probability event and that the righthand sides of~\eqref{pretv2}-\eqref{pretv3} are small. 
In certain cases, the choice of $\epsilon_2$ and $\bdelta_3$ is immediate.  Specifically, suppose $\nll -\nllinfty$ is a linear function of $\theta$. This occurs for generalized linear models, as we discuss in Section~\ref{subsec:log:setup}. This structure has been denoted ``stochastically linear" by~\cite{spokoiny2022SLS}. The implication of this linearity is that $\nabla^k\nll=\nabla^k\nllinfty$ for $k=2,3$. As a result, we can take $\epsilon_2=0$ and any $\bdelta_3(s)$ such that
\beq\label{bdel3-SL-0}
 \sup_{\theta,\theta'\in\bU(s)}\frac{\|\nabla^2\nllinfty(\theta)-\nabla^2\nllinfty(\theta')\|_\F}{|\theta-\theta'|_\F}\leq \bdelta_3(s).\eeq The events $E_2(0)$ and $E_3(s)$ are then trivially satisfied with probability 1. Thus the event $E(s,0)$ reduces to $E(s,0)=E_1(s)\cap E_2(0)\cap E_3(2s)=E_1(s)$.
Note that if $\nll\in C^3(\Theta)$ then we can take
\beq\label{bdel3-SL}
\bdelta_3(s)= \sup_{\theta\in\bU(s)}\|\nabla^3\nllinfty(\theta)\|_\F.\eeq
\begin{lemma}[Nonasymptotic BvM: preliminary lemma under stochastic linearity]\label{prop:lapMLE:SL}Suppose $\nll-\nllinfty$ is a linear function of $\theta$. Let $\bdelta_3$ be any function satisfying~\eqref{bdel3-SL-0}, and assume~\eqref{bUcond}. Then on the event $E_1(s)$, the function $\nll$ has a unique global minimizer $\MLE$, which satisfies $|\MLE-\ground|_\F\leq s\sqrt{d/n}$. Moreover, 
\begin{align}
\TV(\pi_\nll, \gamma_\nll)&\les \bdelta_3(2s)\,\frac{d}{\sqrt n}+e^{-ds^2/36},\label{pretv1-SL}\\
\TV(\gamma_\nll,\bvmgam) &\les s\bdelta_3(s)\frac{d}{\sqrt n}.\label{pretv3-SL}
\end{align} 
for an absolute constant $C$. If in addition $ \bdelta_{01}(2s)\leq\sqrt{nd}/6$, then~\eqref{pretv2} also holds on $E_1(s)$.
\end{lemma}

\begin{remark}[Linear regression setting]\label{rk:quad}In the linear regression model, we have $Y=\Phi\theta + \frac{1}{\sqrt n}\epsilon$, where $Y\in\R^n$ is all the data arranged in a column vector, $\Phi\in\R^{n\times d}$, $\theta\in\R^d$, and $\epsilon\sim\mathcal N(0, I_n)$. In this model, $\nll$ is a quadratic function of $\theta$, and $\nabla^2\nll=\nabla^2\nllinfty\equiv\F$. Thus we have $\pi_\nll=\gamma_\nll=\bvmgam$. We see that $\TV(\pi_v,\bvmgam)=\TV(\pi_v,\pi_\nll)$, and we can now use the bound from~\eqref{pretv2}. Thus in this case, we can expect the BvM to hold under much weaker conditions than $d^2\ll n$. Indeed, the work~\cite{bontemps2011bvm} shows that under certain conditions on the prior, $d\ll n$ suffices; see the discussion following Theorem 2 of that work. Note that~\cite{bontemps2011bvm} differs from our setting in that the author allows for misspecification.\end{remark}

\subsection{Discussion}\label{sec:bvm:discuss}


\paragraph{Assumptions on the prior.} In this discussion we take an asymptotic viewpoint and consider the conditions under which the prior contribution~\eqref{pretv2} can be made arbitrarily small as $n\to\infty$. By definition, all of the above quantities (most notably $d=d_n$) are now indexed by $n$ as $n\to\infty$. However, for simplicity of notation, we omit the $n$ subscript. 


Now, we would like for the first term on the righthand side of~\eqref{pretv2} to go to zero as $n\to\infty$ for each fixed $s$. Meanwhile, the quantity $M_{0}^*$ in the exponent of the second term should remain bounded as $n\to\infty$. If these two conditions hold, we can then take $s\to\infty$ in order to make the righthand side go to zero and the probability of the event $\bar E(s,\epsilon_2)$ go to 1. We prove the latter fact in our examples, with a proper choice of $\epsilon_2$. Thus to summarize, we should have that
\beq\label{priorassum}M_{0}^*=\mathcal O(1),\qquad \bdelta_{01}(2s)= o(\sqrt n),\qquad\text{as}\; n\to\infty.\eeq 
 The condition $M_{0}^*=\mathcal O(1)$ is standard and has been assumed in BvM proofs in the works~\cite{ghosal1999asymptotic,ghosal2000,belloni2014posterior,lu2017bernstein,spokoiny2013bernstein}. Regarding the regularity of $\log\prior$, we impose the slightly stronger assumption that $\log\prior$ is $C^1$ in a neighborhood of $\ground$, whereas the cited works only require that $\log\prior$ is Lipschitz in this neighborhood. However, suppose $\log\prior$ \emph{is} $C^1$, and we use $\bdelta_{01}(2s)$ as an upper bound on this Lipschitz constant. Then we have actually relaxed the condition on this quantity, from $\bdelta_{01}(2s)= o(\sqrt {n/d})$ in the five cited works to $\bdelta_{01}(2s)=o(\sqrt n)$ here. Strengthening the regularity from Lipschitz to $C^1$ allows us to take advantage of the log-Sobolev inequality in $\bU(s)$ to obtain a tighter bound on $\TV(\pi_{\nll},\pi_{v})$ than that of prior works.  \\
  
To get a better sense for the quantities $\Cpri^*$ and $\bdelta_{01}(s)$ and particularly their magnitude, we compute them for several priors.
\begin{example}\label{ex:prior}
\emph{Flat prior}: $\prior(\theta)\equiv1$. Then $\Cpri^*= 0$, $\bdelta_{01}(s)\equiv 0.$\\
\emph{Gaussian prior} $\prior=\mathcal N(\mu, \Sigma)$:
\beqs\label{gpri}
\Cpri^* &\leq\frac1{2d}\|\Sigma^{-1}\|_{\F}|\mu-\ground|_\F^2,\\
 \bdelta_{01}(s)&\leq \|\Sigma^{-1}\|_\F(|\mu-\ground|_\F+s\sqrt{d/n}).
\eeqs
\emph{Multivariate Student's t prior} $\prior=t_\nu(\mu, \Sigma)$:
\beqs\label{tpri}
\Cpri^* &\leq \frac{\nu+d}{2\nu d}\|\Sigma^{-1}\|_\F|\mu-\ground|_\F^2,\\
\bdelta_{01}(s)&\leq \frac{\nu+d}{\nu}\|\Sigma^{-1}\|_\F(|\mu-\ground|_\F+s\sqrt{d/n}).
\eeqs
See Appendix~\ref{app:towardsbvm} for these calculations. Now suppose $s=\mathcal O(1)$ and consider the Gaussian prior $\mathcal N(\mu,\Sigma)$. If, for example, we have $$|\mu-\ground|_\F=\mathcal O(1),\qquad \|\Sigma^{-1}\|_{\F}=\mathcal O(d),$$ then~\eqref{gpri} gives $\Cpri^*=\mathcal O(1)$ and $\bdelta_{01}(s)=\mathcal O(d)$. Now, since we should have $d/\sqrt n=o(1)$ for our bounds in Lemma~\ref{prop:lapMLE} to be small, it follows that $\bdelta_{01}(s)=o(\sqrt n)$. Therefore, the conditions~\eqref{priorassum} are satisfied. For the Student's t prior $t_\nu(\mu, \Sigma)$ with $\nu=\mathcal O(1)$ degrees of freedom,  if
$$|\mu-\ground|_\F=\mathcal O(1),\qquad \|\Sigma^{-1}\|_{\F}=\mathcal O(1),$$ then~\eqref{tpri} gives $\Cpri^*=\mathcal O(1)$ and $\bdelta_{01}(s)=\mathcal O(d)=o(\sqrt n)$, satisfying the conditions~\eqref{priorassum}.

\end{example} 

\paragraph{The condition $d^2\ll n$.}
Here we discuss how the ``universal" factor $d/\sqrt n$ arises in our bounds, and the obstacles to showing that $d^2\ll n$ is a \emph{necessary} condition for the BvM to hold. For simplicity of discussion, we focus on the stochastically linear case. Thus $\nabla^2\nll=\nabla^2\nllinfty$ and therefore in particular, $\epsilon_2=0$. Also, we assume a flat prior, so that $\pi_v=\pi_\nll$.

Consider the preliminary bounds~\eqref{pretv1} and~\eqref{pretv3} in Lemma~\ref{prop:lapMLE}.  The inequality~\eqref{pretv1} is a bound on $\TV(\pi_\nll,\gamma_\nll)$, the TV error of the LA to $\pi_\nll$. It has been shown in~\cite{katskew} that $d/\sqrt n$ is the sharp rate of approximation for the LA; see that work for the intuition behind this rate. Interestingly, the bound~\eqref{pretv3} on $\TV(\gamma_\nll,\bvmgam)$ is of the same order of magnitude as the LA error bound. Thus our overall bound on $\TV(\pi_\nll,\bvmgam)$ is made up of two dominant contributions of order $d/\sqrt n$, stemming from both $\TV(\pi_\nll,\gamma_\nll)$ and $\TV(\gamma_\nll,\bvmgam)$.

Let us take a closer look at how $d/\sqrt n$ arises in the bound on $\TV(\gamma_\nll,\bvmgam)$. Using that $\nabla^2\nll=\nabla^2\nllinfty$, we have $\gamma_\nll=\mathcal N(\MLE, (n\nabla^2\nllinfty(\MLE))^{-1})$ and $\bvmgam=\mathcal N(\MLE, (n\nabla^2\nllinfty(\ground))^{-1})$. 
Now, using a result of~\cite{devroye2018total} and a calculation in Lemma~\ref{lma:TVgauss}, we have
\beq\label{TVD123}
\TV\l(\mathcal N(\mu, \Sigma_1),\;\mathcal N(\mu, \Sigma_2)\r) \asymp \|\Sigma_1^{-1/2}\Sigma_2\Sigma_1^{-1/2}-I_d\|_{\text{Fro}}\leq \frac{\sqrt d}\tau\|\Sigma_2^{-1}-\Sigma_1^{-1}\|_{\Sigma_1^{-1}},\eeq where $\tau>0$ is such that $\Sigma_2^{-1}\succeq \tau\Sigma_1^{-1}$. Here, $a\asymp b$ means $ca\leq b\leq Ca$ for absolute constants $0<c<C$. The $\sqrt d$ in this inequality arises by upper bounding the Frobenius norm by an operator norm. In our case, we apply this result with $\Sigma_1^{-1}=n\nabla^2\nllinfty(\ground)=n\F$ and $\Sigma_2^{-1}=n\nabla^2\nllinfty(\MLE)$. Now on the event $\bar E(s,\epsilon_2)$,  $\tau$ can be taken to be an absolute constant, and we have
$$
\|\Sigma_2^{-1}-\Sigma_1^{-1}\|_{\Sigma_1^{-1}}=\|\nabla^2\nllinfty(\MLE)-\nabla^2\nllinfty(\ground)\|_{\F}\leq\bdelta_3(s)s\sqrt{d/n},
$$ by the definition of $\bar E(s,\epsilon_2)$. Thus overall we get
\beq\label{TVd}
\TV(\gamma_\nll,\bvmgam)\les \sqrt d\|\nabla^2\nllinfty(\MLE)-\nabla^2\nllinfty(\ground)\|_{\F} \leq \sqrt d\l(\bdelta_3(s)s\sqrt{d/n}\r).
\eeq 
We see that the $d/\sqrt n$ comes from the $\sqrt d$ prefactor and the fact that $\MLE$ is at distance $\sqrt{d/n}$ away from $\ground$. Whether or not this upper bound is tight depends on a number of factors, such as whether upper bounding the Frobenius norm by $\sqrt d$ times the operator norm is tight. We leave this question to future work.

Finally, we consider the issue of deriving a lower bound on $\TV(\pi_\nll,\bvmgam)$. A natural approach would be to leverage the lower bound on the LA error $\TV(\pi_\nll,\gamma_\nll)$ derived in~\cite{katskew} via the following inequality:
$$\TV(\pi_\nll,\bvmgam)\geq\TV(\pi_\nll,\gamma_\nll)-\TV(\gamma_\nll,\bvmgam).$$ However, it is not currently possible to use such a technique. This is because the lower bound on $\TV(\pi_\nll,\gamma_\nll)$ and the upper bound on $\TV(\gamma_\nll,\bvmgam)$ are both of the order $d/\sqrt n$. It may be possible to directly argue that $\TV(\pi_\nll,\bvmgam)\gtrsim \TV(\pi_\nll,\gamma_\nll)$, which would allow us to take advantage of the Laplace lower bound. We leave this question to future work.\\

\section{BvM for generalized linear models and exponential families}\label{sec:log}
In this section, we prove the BvM for generalized linear models (GLMs), which encompasses the i.i.d. exponential family setting. We describe the set-up in Section~\ref{subsec:log:setup} and prove the BvM in Section~\ref{bvm:gen:glm}; see also the end of the latter section for a comparison with the literature. In Section~\ref{subsec:logconc}, we specialize the general result to the case of a log-concave exponential family. In Section~\ref{subsec:logreg}, we specialize the result to the case of logistic regression with Gaussian design.
\subsection{Set-up}\label{subsec:log:setup}The general setting we consider is a GLM, in which we observe feature-label pairs $(X_i, Y_i)$, $i=1,\dots,n$. Here, $X_i\in\R^{d\times k}$ is a matrix whose columns constitute $k$ feature vectors in $\R^d$ associated to sample $i$. The vector $Y_i\in\mathcal Y\subseteq\R^k$ is the $k$-variate ``label" corresponding to $X_i$. Given $X_i$ and a parameter vector $\theta\in\R^d$, the model for the distribution of $Y_i$ is
\beq\label{modelb}Y_i\mid X_i \sim p(\cdot\mid X_i^\tr\theta)d\mu,\qquad i=1,\dots,n\eeq for some base measure $\mu$ supported on $\mathcal Y$. Here, $p$ is a $k$-parameter full, minimal, regular exponential family of the form
\begin{align}
p(y\mid\omega) &= \e\l(\omega^\tr y - \psi(\omega)\r),\label{pGLM}\\
\psi(\omega)&=\log\int_{\mathcal Y} e^{\omega^\tr y}d\mu(y).\label{def-psi}\end{align}
\begin{example}[Reduction to i.i.d. exponential family]
As an important special case, if $X_i=I_d$ for all $i=1,\dots, n$ (so in particular, $k=d$) then the GLM reduces to the following i.i.d. exponential family setting:
\beq\label{model-expfam}
Y_i\iid p(\cdot\mid \theta)d\mu,\qquad i=1,\dots,n,
\eeq where $p$ is as in~\eqref{pGLM}.
\end{example}
Because the exponential family is full, minimal, and regular, the domain $\Omega$, given by 
\beq\label{omega}
\Omega=\l\{\omega\in\R^k\; : \; \int_{\mathcal Y} e^{\omega^\tr y}\mu(dy)<\infty\r\},
\eeq 
is known to be convex and open. Moreover, $\psi$ from~\eqref{def-psi} is strictly convex and infinitely differentiable in $\Omega$~\cite[Chapters 7,8]{barndorff2014information}. The domain $\Omega$ gives rise to a domain of valid $\theta$ values, which depends on the chosen features $X_i$:
\beq\label{Theta-def}\Theta = \{\theta\in\R^d\; : \; X_i^\tr\theta\in\Omega\quad\forall i=1,\dots,n\}.\eeq That $\Omega$ is open and convex in $\R^k$ implies $\Theta$ is open and convex in $\R^d$. 
\begin{example}In the i.i.d. exponential family setting, we simply have $\Theta=\Omega\subseteq\R^d$.\end{example}
\begin{example}In Poisson, logistic, and binomial regression we have $\Omega=\R$  and hence $\Theta=\R^d$. Similarly, in multinomial logistic regression we have $\Omega=\R^k$ and hence $\Theta=\R^d$ as well. When $p$ is the exponential distribution, which is a one parameter ($k=1$) exponential family, we have $\Omega=(-\infty,0)$. Thus in this case $\Theta$ is given by the intersection of the $n$ half-spaces $\{\theta\in\R^d\; :X_i^\tr\theta<0\}$.
\end{example}
The model~\eqref{modelb} leads to the following normalized negative log likelihood $\nll$:
\beqs\label{nlldef0}
\nll(\theta)&= \frac1n\sum_{i=1}^n\l[\psi(X_i^\tr \theta) - Y_iX_i^\tr \theta\r].\eeqs The derivatives of $\nll$ are given as follows. Below, we treat $\nabla\psi$ as a column vector in $k$ dimensions.
\beqs\label{nlldef}
\nabla\nll(\theta)&=\frac1n\sum_{i=1}^nX_i\l[\nabla\psi(X_i^\tr \theta)- Y_i\r],\\
\nabla^2\nll(\theta)&=\frac1n\sum_{i=1}^nX_i\nabla^2\psi(X_i^\tr \theta)X_i^\tr,\\
\la\nabla^3\nll(\theta), u^{\otimes3}\ra&=\frac1n\sum_{i=1}^n\la\nabla^3\psi(X_i^\tr\theta), (X_i^\tr u)^{\otimes3}\ra.
\eeqs
We now check Assumption~\ref{A1}. Since $\psi\in C^\infty(\Omega)$ and $\psi$ is convex, we see from~\eqref{nlldef0} that $\nll\in C^\infty(\Theta)$ and $\nll$ is convex with probability 1. Furthermore, we see from~\eqref{nlldef} that $\nabla^2\nll=\nabla^2\nllinfty$ since $\nabla^2\nll$ does not depend on the random $Y_i$. Thus it remains to check $\nabla^2\nll(\ground)\succ0$, which we do in the following lemma.
\begin{lemma}If the linear span of the union of all $kn$ columns of the matrices $X_i$, $i=1,\dots,n$ equals $\R^d$, then $\nabla^2\nll(\theta)\succ0$ for all $\theta\in\Theta$. 
\end{lemma}
\begin{proof}
It suffices to show $u^\tr\nabla^2\nll(\theta)u=0$ implies $u=0$. We have
\beqs
u^\tr\nabla^2\nll(\theta)u=\frac1n\sum_{i=1}^nu^\tr X_i\nabla^2\psi(X_i^\tr \theta)X_i^\tr u \geq \min_{i=1,\dots,n}\lambda_{\min}\l(\nabla^2\psi(X_i^\tr \theta)\r)\frac1n\sum_{i=1}^n|X_i^\tr u|^2.
\eeqs Since $\psi$ is strictly convex, we have that $\lambda_{\min}\l(\nabla^2\psi(X_i^\tr \theta)\r)>0$ for all $i$. Hence $u^\tr\nabla^2\nll(\theta)u=0$ implies $X_i^\tr u=0$ for all $i$. This implies $u$ is orthogonal to each column of $X_i$ for all $i=1,\dots, n$. But if the span of the columns is $\R^d$, then $u$ must be zero.
\end{proof}
\subsection{BvM for GLMs}\label{bvm:gen:glm}
In this section, we prove the BvM in the setting described above. Note from~\eqref{nlldef} that the random $Y_i$ appears only in the first derivative $\nabla\nll$. Thus $\nll-\nllinfty$ is linear. As discussed below Lemma~\ref{prop:lapMLE}, this structure leads to some simplifications. In particular, we can apply the more specialized Lemma~\ref{prop:lapMLE:SL}, with $\bdelta_3$ as in~\eqref{bdel3-SL}. It remains only to bound from below the probability of the event $E_1(s)$. First, we record the specific form of the Fisher information matrix $\F$ and $\bdelta_3$ for our model. We have
\beq\label{FdefGLM}
 \F=\nabla^2\nllinfty(\ground) =\frac1n\sum_{i=1}^nX_i\nabla^2\psi(X_i^\tr \ground)X_i^\tr 
\eeq and
\beqs\label{bdelbd}
\bdelta_3(s)&=  \sup_{\theta\in\bU(s),\,u\neq0}\;\frac{\frac1n\sum_{i=1}^n\lla\nabla^3\psi(X_i^\tr \theta),\,(X_i^{\tr}u)^{\otimes3}\rra}{\l(\frac1n\sum_{i=1}^n\lla\nabla^2\psi(X_i^\tr \ground),\,(X_i^{T}u)^{\otimes2}\rra\r)^{3/2}}.
\eeqs 
\begin{example}[Key quantities for i.i.d. exponential family setting]
In the i.i.d. exponential family setup, $\F$ and $\bdelta_3$ reduce to the following: 
\beqs\label{F-bdel3-expfam}
\F=\nabla^2\psi(\ground), \qquad \bdelta_3(s)= \sup_{\theta\in\bU(s),\,u\neq0}\frac{\lla\nabla^3\psi(\theta),\,u^{\otimes3}\rra}{\lla\nabla^2\psi(\ground),u^{\otimes2}\rra^{3/2}}.\eeqs
\end{example}

 The following lemma bounds the probability of $E_1(s)$. 
\begin{lemma}\label{bern-gauss}
Let 
$Y_i\mid X_i\sim p(\cdot\mid X_i^\tr\ground)$. Suppose~\eqref{bUcond} is satisfied. Then the event $E_1(s)$
has probability at least $1-\exp(-s^2d/10)$.
\end{lemma}
Let us give some intuition for this result. Both for GLMs and in fact for much more general settings, it holds 
\beq\label{E-var}\E_{Y^n\sim P^n_\ground}[\nabla\nll(\ground)]=\nabla\nllinfty(\ground)=0,\qquad \Var_{Y^n\sim P^n_\ground}[\nabla\nll(\ground)]= \frac1n\F.\eeq We have used the more general notation from Section~\ref{subsec:bayes}. In the setting of GLMs, we have $Y^n=(Y_1, \dots, Y_n)$, and $P^n_\ground=\otimes_{i=1}^np(\cdot\mid X_i^\tr\ground)d\mu$. Using~\eqref{E-var}, we find that
\beq
\sqrt n\|\nabla\nll(\ground)\|_{\F} = |(\F/n)^{-1/2}\nabla\nll(\ground)| = \l|\Var(\nabla\nll(\ground))^{-1/2}\l(\nabla\nll(\ground)-\E\nabla\nll(\ground)\r)\r|,
\eeq
where the expectation and variance are with respect to $P^n_\ground$. Now, recalling~\eqref{nlldef}, we see that $\nabla\nll(\ground)$ is given by an average of $n$ independent random variables. Therefore, the Central Limit Theorem suggests $Z_n:=\Var(\nabla\nll(\ground))^{-1/2}\l(\nabla\nll(\ground)-\E\nabla\nll(\ground)\r)$ is approximately standard normal. Moreover, we can write the event $E_1(s)$ as $E_1(s)=\{|Z_n|\leq s\sqrt d\}$. Thus the probability of $E_1(s)^c$ should scale as a Gaussian tail probability and indeed, Lemma~\ref{bern-gauss} shows that $\mathbb P(|Z_n|\geq s\sqrt d)\leq e^{-s^2d/10}$.

The asymptotic distribution of the norm of a standardized sample mean was first considered in~\cite{portnoy1988asymptotic}. Later,~\cite[Section F.3]{spokoiny2023inexact} proved nonasymptotic tail bounds on the norm of a sub-Gaussian or sub-exponential random vector. To be self-contained we give our own proof of Lemma~\ref{bern-gauss}. The proof does not actually rely on the CLT; we use an $\epsilon$-net argument and apply Chernoff's inequality. \\

We now conclude the BvM by a direct application of Lemmas~\ref{bern-gauss} and~\ref{prop:lapMLE:SL}.
\begin{prop}[Non-asymptotic BvM for GLMs]\label{lma:bvmlog}Suppose the columns of the matrices $X_i$, $i=1,\dots,n$ span $\R^d$, let $\Theta$ be as in~\eqref{Theta-def}, $\F$ as in~\eqref{FdefGLM}, and $\bdelta_3$ be as in~\eqref{bdelbd}. Suppose that for some $s\geq12$, the conditions from~\eqref{bUcond} are satisfied, and $\bdelta_{01}(2s)\leq\sqrt{nd}/6$. Then on an event of probability at least $1-\exp(-ds^2/10)$ with respect to the distribution of the $Y_i$, we have
\beq\label{lma:bvmlog:eq}
\TV(\pi_v,\bvmgam) \les \frac{\bdelta_{01}(2s)}{\sqrt{ n}}+s\bdelta_3(2s)\frac{d}{\sqrt n}+e^{d(\Cpri^*-(s/12)^2)}.
\eeq 
\end{prop}
We now prove the traditional asymptotic BvM. We assume that $d=d_n$ and $k=k_n$ may change with $n$, and for each $n$ we are given $n$ matrices $X_{i,n}\in\R^{d_n\times k_n}$, $i=1,\dots,n$, and a sequence of functions $\psi=\psi_n$ in the exponential family model~\eqref{pGLM}. This induces the following $n$-dependent quantities: $\Omega=\Omega_n\subseteq\R^{k_n},\Theta=\Theta_n\subseteq\R^{d_n}$, $\F=F^*_n$, $\bdelta_3=\bdelta_{3n}$, defined as in Section~\ref{subsec:log:setup}. In the definition~\eqref{bdelbd} of $\bdelta_3$, note that the local neighborhood $\bU=\bU_n$ also changes with $n$. Also, write $\Cpri=M_{0n}^*$ and $\bdelta_{01}=\bdelta_{01n}$, which are as in Definition~\ref{def:Ugen} for each $n$. Finally, we write $\pi_{v_n},\bvmgam_n$ to emphasize the dependence on $n$ of the posterior and the Gaussian in the BvM.

To obtain the asymptotic BvM from the bound~\eqref{lma:bvmlog:eq}, we take the $n\to\infty$ limit for each fixed $s$, and then take $s\to\infty$. Note that if $d_n\to\infty$ as $n\to\infty$, then the second step of taking $s\to\infty$ is not necessary. 
\begin{corollary}[Asymptotic BvM for GLMs]\label{corr:bvmasymexp}
Suppose the linear span of the $k_nn$ columns of the matrices $X_{i,n}$, $i=1,\dots,n$ equals $\R^{d_n}$, and that $M_{0n}^*=\mathcal O(1)$ as $n\to\infty$. Also, for each fixed $s\geq0$, suppose the following hold: (1) the neighborhood $\bU_n(s)$ is contained in $\Theta_n$ when $n$ is large enough, (2) $\bdelta_{01n}(s)=o(\sqrt n)$, and (3)
$$\bdelta_{3n}(s)\frac{d_n}{\sqrt n} = o(1).$$ Then $\TV(\pi_{v_n}, \bvmgam_n)\to0$ with probability tending to 1 as $n\to\infty$.
\end{corollary} 
This result is as explicit as we can hope for in the very general setting of an arbitrary GLM. Given a particular GLM of interest, it remains to bound the quantity $\bdelta_3$ from~\eqref{bdelbd}. In the following two subsections, we consider two special cases in which $\bdelta_3$ can either be simplified or explicitly bounded. In the first case, a log-concave exponential family, we show that the third derivative of $\psi$ can be bounded in terms of the second derivative. This leads to a simple bound on $\bdelta_3$ avoiding operator norms of $d\times d\times d$ tensors. In the second case, logistic regression with Gaussian design, we show the numerator and denominator in~\eqref{bdelbd} are upper- and lower-bounded by absolute constants, respectively. This holds with high probability over the random design.\\

Let us compare Proposition~\ref{lma:bvmlog} to prior work. The work~\cite{spokoiny2013bernstein} proves a BvM for generalized linear models in $k=1$, though the conditions on the design under which the BvM holds are left unspecified. The special case of exponential families was studied in~\cite{ghosal2000} and~\cite{belloni2014posterior}. As we show in Appendix~\ref{subsec:bell-spok}, an explicit TV bound can be recovered from the proof of~\cite{belloni2014posterior}, and cast as a product of a model dependent and universal factor. The work~\cite{spokoiny2013bernstein} does (essentially) state an explicit TV bound, which can also be brought into this same form. Once this has been done, we see that our BvM tightens the dimension dependence of the universal factor from $\sqrt{d^3/n}$ to $\sqrt{d^2/n}$, while our model-dependent factor $\bdelta_3$ is very analogous to that of~\cite{spokoiny2013bernstein} and~\cite{belloni2014posterior}

The work~\cite{belloni2014posterior} in turn improves on~\cite{ghosal2000} by removing a $\log d$ factor.

\subsection{Application to log concave exponential families}\label{subsec:logconc}
Similarly to~\cite{belloni2014posterior}, we now consider the case of a log concave exponential family. Recall that in the i.i.d. exponential family setting, we have $k=d$ and $X_i=I_d$ for all $i=1,\dots,n$, so that 
\beq\label{modelb-expfam} Y_i\iid p(\cdot\mid\theta)d\mu,\qquad i=1,\dots, n.\eeq Furthermore, using standard properties of exponential families, we can write the derivatives of $\psi$ in terms of moments of the distribution~\eqref{pGLM}. In particular,
\beqs\label{mom-form}
\nabla\psi(\theta)&=\E_\theta[Y_1],\\
\nabla^2\psi(\theta)&=\E_\theta\l[(Y_1-\E_\theta[Y_1])(Y_1-\E_\theta[Y_1])^\tr\r]=\Var_\theta(Y_1),\\
\nabla^3\psi(\theta) &=\E_{\theta}\l[(Y_1-\E_\theta[Y_1])^{\otimes3}\r].
\eeqs Here, $\E_\theta[\cdot]$ is shorthand for $\E_{Y_1\sim p(\cdot\mid \theta)}[\cdot]$, and similarly for $\Var_\theta(\cdot)$.

Let us now consider the key quantity $\bdelta_3$ in the BvM, which for exponential families takes the form given in~\eqref{F-bdel3-expfam}. Using the above relationship between derivatives of $\psi$ and centered moments of $Y_1$, we can write $\bdelta_3$ as follows:
\beq\label{F-bdel3-expfam-2}
\bdelta_3(s)= \sup_{\theta\in\bU(s),\,u\neq0}\;\frac{\E_\theta\l[\l(u^\tr Y-u^\tr \E_\theta[Y]\r)^3\r]}{\Var_\ground(u^\tr Y)^{3/2}}.
\eeq

Thus we see that we need to bound a ratio of third and second moments of the distribution $p(\cdot|\theta)$ from~\eqref{pGLM}. When this distribution is log concave, it is well known that the third moment can be bounded above in terms of the second moment. This allows us to prove the following lemma.
\begin{lemma}\label{lma:bdel23}Let $F(\theta)=\Var_\theta(Y)$ and define $$\bdelta_2(s)=\sup_{\theta\in\bU(s)}\|F(\ground)^{-1/2}F(\theta)F(\ground)^{-1/2}\|=\sup_{\theta\in\bU(s)}\|F(\theta)\|_\F,$$ where $F(\theta)=\Var_\theta(Y)=\nabla^2\psi(\theta)$. 
If the base distribution $\mu$ from~\eqref{modelb-expfam} has a log concave density, then
$$\bdelta_3(s)\leq C_{23}\bdelta_2(s)^{3/2}$$ for some absolute constant $C_{23}$, where $\bdelta_3$ is defined as in~\eqref{F-bdel3-expfam-2}. 
\end{lemma}
\begin{proof}If $\mu$ has a log concave density then so does $p(\cdot|\theta)$, for each $\theta\in\Theta$. Now, log concavity is preserved under affine transformations ~\cite[Section 3.1.1]{saumard2014log}. Therefore if $Y\sim p(\cdot\mid\theta)$ then $u^\tr Y-\E_\theta[u^\tr Y]$ also has a log concave distribution for all $u\in\R^d$. 
\cite[Theorem 5.22]{lovasz2007geometry} now gives that
$$\E_\theta\l[\l|u^\tr Y-\E_\theta[u^\tr Y]\r|^3\r]\leq C_{23}\Var_\theta(u^\tr Y)^{3/2}=C_{23}\l(u^\tr F(\theta)u\r)^{3/2}.$$ Taking the supremum over all $u$ such that $|u|_\F^2=1$ and $\theta\in\bU(s)$ concludes the proof.
\end{proof}
Using this result in Proposition~\ref{lma:bvmlog} leads to the following simplified BvM. 
\begin{prop}[Nonasymptotic BvM for log concave exponential families]\label{corr:bvmexp:logc}Suppose the base density $h$ of the exponential family~\eqref{pGLM} is log concave. Suppose that for some $s\geq12$ it holds 
\beqs\label{bdel2assump}\bdelta_2(2s)=&\sup_{\theta\in\bU(2s)}\|F(\ground)^{-1/2}F(\theta)F(\ground)^{-1/2}\|\leq 3/2,\\
s\sqrt{d/n}\leq &(16C_{23})^{-1},\qquad \bdelta_{01}(2s)\leq\sqrt{nd}/6,\qquad \bU(2s)\subset\Theta.\eeqs

 Then on an event of probability at least $1-\exp(-ds^2/10)$ with respect to the distribution of the $Y_i$, we have
\beqs\label{bvm-highprob-exp-logc}
\TV(\pi_v,\bvmgam) \les \frac{\bdelta_{01}(2s)}{\sqrt{ n}}+s\frac{d}{\sqrt n}+e^{d(\Cpri^*-(s/12)^2)}.
\eeqs 
\end{prop}
\begin{proof}
The assumptions and Lemma~\ref{lma:bdel23} give  $\bdelta_3(2s)\leq (3/2)^{3/2}C_{23} \leq 2C_{23}$. Thus in particular $s\bdelta_3(2s)\sqrt{d/n} \leq 2C_{23}s\sqrt{d/n} \leq 1/8.$ Therefore the conditions in~\eqref{bUcond} are satisfied. Substituting $\bdelta_3(2s)\les1$ in the righthand side of~\eqref{lma:bvmlog:eq} in Proposition~\ref{lma:bvmlog} concludes the proof. \end{proof}

In the above proposition and proof, it may seem like we simply imposed a strong assumption --- that $\bdelta_2(2s)$ is bounded by an absolute constant --- in order to conclude that $\bdelta_3(2s)$ is also bounded by an absolute constant, using the log concavity. However, we claim that the assumption $\bdelta_2(2s)\leq3/2$ is actually \emph{weaker} than the assumption $2s\bdelta_3(2s)\sqrt{d/n}\leq1/4$ we have been using in previous results, including Proposition~\ref{lma:bvmlog}. Indeed, to see this, first note that $\bdelta_2(0)=1$. It then follows by Taylor's theorem that
$$\bdelta_2(2s)\leq 1 + \sup_{\theta\in\bU(2s)}\|\nabla^2\psi(\theta)-\nabla^2\psi(\ground)\|_\F\leq 1 + \bdelta_3(2s)2s\sqrt{d/n}\leq\frac32.$$

Thus, if one prefers to start with the stronger condition $2s\bdelta_3(2s)\sqrt{d/n}\leq1/4$, we have shown that for log concave families,
$$
\bdelta_3(2s)\leq\frac{1}{8s}\sqrt{n/d}\implies \bdelta_3(2s) \leq C=(3/2)^{3/2}C_{23}.
$$

\subsection{Application to logistic regression with random design}\label{subsec:logreg}
We now apply Proposition~\ref{lma:bvmlog} to logistic regression. In this model, $k=1$ and the $Y_i$ are binary. In other words, $X_i\in\R^d$ and $Y_i\in\{0,1\}$ is the corresponding label. The distribution of the random variables $Y_i$ given the $X_i$ is
$$Y_i\mid X_i\sim\mathrm{Bernoulli}(\psi'(X_i^\tr\theta)),\qquad \psi(\omega):=\log(1+e^\omega).$$ The probability mass function of this Bernoulli random variable can be written in the form~\eqref{pGLM}, with $\psi$ as above.

Now, the statement of the BvM in Proposition~\ref{lma:bvmlog} is conditional on the $X_i$'s. But by specifying a distribution on the $X_i$'s, we can gain insight into the ``typical" size (typical with respect to this distribution) of the key quantity $\bdelta_3(2s)$ in the bound~\eqref{lma:bvmlog:eq}. For the distribution, we choose the standard example of i.i.d. Gaussian design:
$$X_i\iid \mathcal N(0, I_d),\quad i=1,\dots,n.$$ The case of a non-identity covariance matrix can be similarly handled by a linear transformation. The following two lemmas will allow us to bound $\bdelta_3(2s)$ with high probability with respect to the design.
\begin{lemma}[Adaptation of Lemma 7, Chapter 3,~\cite{pragyathesis}]\label{lma:pragya}
Suppose $d<n/2$. Then  for some $\lambda=\lambda(|\ground|)>0$ depending only on $|\ground|$, it holds
$$\mathbb P(\F\succeq\lambda I_d)=\mathbb P\l(\frac1n\sum_{i=1}^n\psi''(X_i^\tr \ground)X_iX_i^\tr \succeq\lambda I_d\r)\geq 1-4e^{-Cn},$$ 
where $C$ is an absolute constant (independent of $\lambda$). 
\end{lemma} The function $|\ground|\mapsto\lambda(|\ground|)$ is nonincreasing, so if $|\ground|$ is bounded above by a constant then $\lambda=\lambda(|\ground|)$ is bounded away from zero by a constant.
\begin{lemma}\label{lma:adam}
If $d\leq n$ then there are absolute constants $C,C'$ such that 
\beqs\label{c34logreg}
\mathbb P\l(\sup_{|u|=1}\frac1n\sum_{i=1}^n|u^\tr X_i|^3 \leq C\l(1+\frac{d^{3/2}}{n}\r)\r)\geq1-4e^{-C'\sqrt n}.
\eeqs 
\end{lemma}
The lemma follows almost immediately from Theorem 4.2 in~\cite{adamczak2010quantitative}; see the appendix for the short calculation. We now combine Lemma~\ref{lma:pragya}, Lemma~\ref{lma:adam}, and the fact that $\|\psi'''\|_\infty:=\sup_{t\in\R}|\psi'''(t)|<\infty$, to derive the following bound on $\bdelta_3(2s)$:
\beqs\label{bdel3-log}
\bdelta_3(2s) &= \sup_{\theta\in\bU(2s),\,u\neq0}\;\frac{\frac1n\sum_{i=1}^n\psi'''(X_i^\tr \theta)(X_i^{\tr}u)^{3}}{\l(\frac1n\sum_{i=1}^n\psi''(X_i^\tr \ground)(X_i^{T}u)^{2}\r)^{3/2}}\\
&\leq \|\psi'''\|_\infty\lambda^{-3/2}\sup_{u\neq0}\frac1n\sum_{i=1}^n|X_i^{\tr}u|^{3}/\|u\|^3\\
&\leq C \|\psi'''\|_\infty\lambda^{-3/2}\l(1+d^{3/2}/n\r).
\eeqs
This bound holds on an event of probability at least $1-8e^{-C\sqrt n}$ for some new constant $C$. On this same event, we also have
\beq\label{bdel01-log}
\bdelta_{01}(s) \leq \lambda^{-1/2}\bar\delta^*_{01}\l( \lambda^{-1/2}s\r),
\eeq where
$$
\bar\delta^*_{01}(t) = \sup_{|\theta-\ground|\leq t\sqrt{d/n}}|\nabla\log\prior(\theta)|.
$$
Using the bounds~\eqref{bdel3-log} and~\eqref{bdel01-log} in Proposition~\ref{lma:bvmlog}, we conclude the following non-asymptotic BvM.
\begin{corollary}[Non-asymptotic BvM for logistic regression with random design]\label{corr:bvmlogreg}
Suppose $|\ground|\leq C_0$ and $d^{3/2}\leq C_0n$ for an absolute constant $C_0$. There exists a constant $C_1$ depending only on $C_0$ and an absolute constant $C_2$ such that if
$$
s\geq12, \qquad C_1s\sqrt{d/n}\leq 1,\qquad C_1\bar\delta^*_{01}(C_1s)\leq \sqrt{nd}
$$
then
\beqs\label{TV-log}
\TV(\pi_v, \bvmgam)&\les \frac{\bar\delta^*_{01}(C_1s)}{\sqrt n} +\frac{sd}{\sqrt n}+ e^{d[\Cpri^*  - (s/12)^2]}
\eeqs 
with probability at least  $1-8e^{-C_2\sqrt n}-e^{-ds^2/10}$ with respect to the joint feature and label distribution. The suppressed constant in~\eqref{TV-log} depends only on $C_0$.
 \end{corollary}We now take $n\to\infty$ and then $s\to\infty$ to prove the following asymptotic BvM.
\begin{corollary}[Asymptotic BvM for logistic regression with random design]\label{corr:bvmasymlog}Let $\theta^*_n\in\R^{d_n}$ be a sequence of ground truth vectors such that $|\theta^*_n|=\mathcal O(1)$ as $n\to\infty$. Write $v_n,\bvmgam_n,M_{0n}^*,\bar\delta^*_{01n}$ to emphasize the dependence of these quantities on $n$. If $M_{0n}^*=\mathcal O(1)$, $\bar\delta_{01n}^{*}(s)=o(\sqrt n)$ for each $s\geq0$, and $$d_n^2/n=o(1),$$ then
$\TV(\pi_{v_n}, \bvmgam_n)\to0$ with probability tending to $1$ as $n\to\infty.$ The probability is with respect to the joint feature-label distribution $X_i\iid\mathcal N(0,I_{d_n})$, $Y_i|X_i\sim\mathrm{Ber}(\psi'(X_i^\tr\theta^*_n))$, $i=1,\dots,n$.\end{corollary} 

To our knowledge, this is the first ever BvM to take into account the randomness of the design in the high-dimensional regime.

\section{Observations of a discrete probability distribution}\label{sec:pmf} In this section, we consider a model in which the parameter of interest are the probabilities of a finite state probability mass function (pmf) $\ground=(\theta_0^*,\theta_1^*,\dots,\theta_d^*)$, and we are given $n$ i.i.d. draws from $\ground$. This is equivalent to observing a single sample $Y \sim \mathrm{Multi}(n, \ground)$, the multinomial distribution with $n$ trials and probabilities $\theta_0^*,\theta_1^*,\dots,\theta_d^*$. A BvM for this model, in which $d$ can grow large, was first proved in~\cite{boucheron2009discrete}. We compare our result to that of~\cite{boucheron2009discrete} at the end of Section~\ref{subsec:pmf:pf}. 

There is a reparameterization which turns the model into an exponential family; the works~\cite{ghosal2000} and~\cite{belloni2014posterior} on BvMs for exponential families both studied it in this form. However, a BvM for one form of the model does not immediately imply the BvM for the other form; as noted in~\cite{boucheron2009discrete}, the effect of reparameterization on asymptotic normality must first be established. 

Proofs of omitted results in this section can be found in Appendix~\ref{app:sec:pmf}.
\subsection{Set-up} The $d+1$ probabilities of a pmf over states $0,1,\dots,d$ add up to 1. We therefore define the parameter space to be the probabilities of states $1,\dots,d$ only:
\beqs\label{Thetapmf}
\Theta = \{\theta=(\theta_1,\dots,\theta_d)\mid 0<\theta_1+\dots+\theta_d<1\}.\eeqs The cases $\theta_1+\dots+\theta_d=0$ and $\theta_1+\dots+\theta_d=1$ are degenerate, so we exclude them. Now, to every point $\theta\in\Theta$, we associate the value
$$\theta_0=1-\sum_{j=1}^d\theta_j=1-\one^\tr \theta.$$ Note that if $\theta\in\Theta$ then $(\theta_0,\dots,\theta_d)$ is a pmf on $d+1$ states such that each probability $\theta_j,j=0,\dots,d$ is strictly positive. We will often abuse notation by interchangeably using $\theta$ to denote either $(\theta_0,\dots,\theta_d)$ or $(\theta_1,\dots,\theta_d)$. One need only remember that $\theta_1,\dots,\theta_d$ are free parameters, while $\theta_0$ is determined from the others. We observe counts $N=(N_0,N_1,\dots, N_d)\sim\mathrm{Multi}(n, \ground)$, where $\ground=(\theta_0^*,\theta_1^*,\dots,\theta_d^*)$ is the ground truth pmf for some $(\theta_1^*,\dots,\theta_d^*)\in\Theta$. For the proportions corresponding to the counts, we use the notation $$\bar N = (\bar N_0,\dots,\bar N_d),\quad \bar N_j := \frac1nN_j.$$ Define also
$$\tmin = \min_{j=0,\dots, d}\theta_j^*,$$ which will play an important role in describing how far $\ground$ is from the boundary of $\Theta$. Now, the multinomial observations give rise to the likelihood $L(\theta)=\prod_{j=0}^d\theta_j^{N_j}$, and negative normalized log likelihood $\nll=-\frac1n\log L$:
\beq\label{nll-pmf}
\nll(\theta)=-\sum_{j=0}^d\bar N_j\log\theta_j = -\bar N_0\log(1-\one^\tr \theta)-\sum_{j=1}^d\bar N_j\log\theta_j .\eeq The first three derivatives of $\nll$, with respect to the free parameters $\theta_1,\dots,\theta_d$, are given as follows:
\beq\label{nllderiv-pmf}
\setlength{\jot}{5pt}
\begin{split}
\nabla\nll(\theta)&= - \bigg(\frac{\bar N_j}{\theta_j}\bigg)_{j=1}^d+ \frac{\bar N_0}{\theta_0}\one,\\
\nabla^2\nll(\theta) &= \mathrm{diag}\bigg(\bigg(\frac{\bar N_j}{\theta_j^2}\bigg)_{j=1}^d\bigg) + \frac{\bar N_0}{\theta_0^2}\one\one^\tr \\
\nabla^3\nll(\theta) &= -2\mathrm{diag}\bigg(\bigg(\frac{\bar N_j}{\theta_j^3}\bigg)_{j=1}^d\bigg) + 2\frac{\bar N_0}{\theta_0^3}\one^{\otimes 3}.
\end{split}
\eeq
Next, recall that $\F=\nabla^2\nllinfty(\ground)=\E[\nabla^2\nll(\ground)]$. Using that $\E[\bar N]=\ground$ gives
\beq\label{pmf-Fdef}
\F=\nabla^2\nllinfty(\ground)= \mathrm{diag}\bigg(\bigg(\frac{1}{\theta_j^*}\bigg)_{j=1}^d\bigg) + \frac{1}{\theta_0^*}\one\one^\tr.
\eeq 
It is clear from~\eqref{nll-pmf} that $\nll\in  C^\infty(\Theta)$. Also, we see from the second equation in~\eqref{nllderiv-pmf} that $\nabla^2\nll(\theta)\succeq0$ for all $\theta\in\Theta$ with probability 1. Hence, $\nll$ is convex on $\Theta$. Finally, since $\ground\in\Theta$ (i.e. all $\theta_j^*$ are strictly positive), we see from~\eqref{pmf-Fdef} that $\F\succ0$. Thus Assumption~\ref{A1} is satisfied for this model.

\begin{remark}\label{rk:MLE:pmf}Suppose $\bar N\in\Theta$, which implies the counts $N_j$, $j=0,\dots,d$ are all strictly positive. Since $\nll$ is convex, any strict local minimizer $\MLE$ of $\nll$ in $\Theta$ must be the unique global minimizer, i.e. the MLE. But we see from the first two equations  in~\eqref{nllderiv-pmf} that $\nabla\nll(\bar N)=0$ and $\nabla^2\nll(\bar N)\succ0$. Thus $\bar N$ is the unique MLE as long as $\bar N\in\Theta$.
\end{remark}

As usual, we define
\beq\label{Upmf}
\bU(r)=\{\theta\in\R^d: |\theta-\ground|_\F \leq r\sqrt{d/n}\}.
\eeq
\begin{defn}[Chi-squared divergence]
Let $\theta=(\theta_0,\dots,\theta_d)$ be a strictly positive pmf on $d+1$ states. Let $\omega$ be either $\omega=(\omega_0,\dots,\omega_d)\in\R^{d+1}$ such that $\omega^\tr \one=1$, or $\omega=(\omega_1,\dots,\omega_d)\in\R^d$. In the latter case, set $\omega_0 = 1-\sum_{j=1}^d\omega_j$. Then we define 
$$\chi^2(\omega||\theta) = \sum_{j=0}^d(\omega_j-\theta_j)^2/\theta_j.$$ \end{defn}
\begin{remark}
We will make frequent use of the bound
\beq\label{max2chi}
\max_{j=0,\dots,d}\l|\frac{\theta_j}{\theta_j^*}-1\r|^2\leq \chi^2(\theta||\ground)/\tmin,
\eeq which follow directly from the above definition of $\chi^2$ divergence.
\end{remark}
We now show that the neighborhoods $\bU(r)$ from~\eqref{Upmf} are $\chi^2$ balls around $\ground$. We also characterize how large $r$ can be to ensure $\bU(r)$ remains a subset of $\Theta$, and describe a useful property of pmfs in $\bU(r)$
\begin{lemma}\label{lma:thetabd}The neighborhoods $\bU(r)$ from~\eqref{Upmf} are equivalently given by $$\bU(r)=\{\theta\in\R^d\; : \;\chi^2(\theta||\ground)\leq r^2d/n\}.$$ If $r^2d/n< \tmin/4$, then $\bU(2r)\subset\Theta$ and $\bU(r)\subset\{\theta\in\Theta\; : \theta_j\geq\theta_j^*/2\;\,\forall j=0,1,\dots,d\}$.
\end{lemma}
\begin{remark}
The lemma shows that how large $r$ can be to ensure $\bU(2r)\subset\Theta$ depends on how small $\tmin$ is. This explains our earlier comment that $\tmin$ controls how far $\ground$ is from the boundary of $\Theta$.
\end{remark}

\subsection{BvM proof}\label{subsec:pmf:pf}In this section, we prove the BvM by applying Lemma~\ref{prop:lapMLE}, though we will not use event $E(s,\epsilon_2)$. 
Instead, we define a different event $E_0(s)$ and show that $E_0(s)\subset \bar E(s,\epsilon_2)$ for an appropriate choice of $\epsilon_2=\epsilon_2(s)$. Specifically, for $s\geq 0$, let 
$$E_0(s):=\l\{\bar N\in\bU(s)\r\} = \l\{\chi^2(\bar N||\ground)\leq s^2d/n\r\} = \l\{|\bar N-\ground|_\F^2\leq s^2d/n\r\}.$$ We prove that if $s$ is larger than a sufficiently large absolute constant and $s^2d/n\tmin$ is smaller than a sufficiently small absolute constant, then:\\

(1) on $E_0(s)$, there exists a unique MLE $\MLE$ satisfying $|\MLE-\ground|_\F\leq s\sqrt{d/n}$, \\

(2) on $E_0(s)$, it holds $\|\nabla^2\nll(\ground)-\F\|_\F\leq \epsilon_2(s):=\sqrt{s^2d/n\tmin}$, \\

(3) on $E_0(s)$, it holds $\sup_{\theta\in\bU(2s)}\|\nabla^3\nll(\theta)\|_{\F} \leq \bdelta_3(2s):= C/\sqrt{\tmin}$, \\

(4) $E_0(s)$ has probability at least $1-e^{-Cs^2d}$.\\

The proof of (1) is immediate. Indeed, if $s^2d/n\leq\tmin$ is small enough then $\bU(s)\subset\Theta$ by Lemma~\ref{lma:thetabd}, and we know that $\bar N\in\bU(s)$ on the event $E_0(s)$. Hence $\bar N\in\Theta$, and Remark~\ref{rk:MLE:pmf} then shows $\MLE=\bar N$ is the unique MLE. Since $\bar N\in \bU(s)$ we also have $|\bar N-\ground|_\F\leq s\sqrt{d/n}$ by definition. See Appendix~\ref{app:sec:pmf} for the proof of (2),(3),(4). Note that the result (3) holds for $s$ small enough; we do not expect a uniform-in-$s$ bound on $\sup_{\theta\in\bU(2s)}\|\nabla^3\nll(\theta)\|_{\F}$. To prove (4) we use a result from~\cite{boucheron2009discrete}'s BvM proof for this model. The result is in turn based on Talagrand's inequality for the suprema of empirical processes~\cite{massart}. 

\begin{remark}\label{rk:pmin}Since the ground truth pmf values add up to 1, it follows that $\tmin\leq 1/(d+1)$. Therefore the upper bound in (3) on the third derivative tensor operator norm scales as $\sqrt d$. This bound is tight. To see this, suppose for simplicity that $\bar N_j=\theta_j^*$ exactly. This allows us to compute $\|\nabla^3\nll(\ground)\|_\F$ explicitly; see Appendix~\ref{app:sec:pmf}. The result is that
\beq
\sup_{\theta\in\bU(2s)}\|\nabla^3\nll(\theta)\|_{\F}\geq\|\nabla^3\nll(\ground)\|_\F = 2\frac{1-2\tmin}{\sqrt{\tmin}\sqrt{1-\tmin}}\sim \frac{1}{\sqrt\tmin}.
\eeq 
\end{remark}

Let us show that (1)-(4) combined with Lemma~\ref{prop:lapMLE} finish the BvM proof. First, provided $s^2d/n\tmin$ is small enough, (1)-(4) gives that $\bU(2s)\subseteq\Theta$, $\epsilon_2\leq 1/2$ and $2s\bdelta_3(2s)\sqrt{d/n}\leq 1/4$. Thus~\eqref{bUcond} is satisfied. Second, (1)-(4) show that $E_0(s)\subset \bar E(s,\epsilon_2)$, and hence the bounds in Lemma~\ref{prop:lapMLE} are satisfied on $E_0(s)$. Finally, in (4) we obtain a lower bound on the probability of $E_0(s)$. Thus the BvM is proved. Substituting $\bdelta_3(2s)= C/\sqrt{\tmin}$ into the bounds from Lemma~\ref{prop:lapMLE}, we obtain the following corollary.
\begin{corollary}\label{corr:bvm-pmf}
Suppose $s$ is larger than a sufficiently large absolute constant, $s^2d/n\tmin$ is smaller than a sufficiently small absolute constant, and $\bdelta_{01}(2s)\leq \sqrt{nd}/6$, where $\bdelta_{01}$ is as in Definition~\ref{def:Ugen} with $\F$ as in~\eqref{pmf-Fdef}. Then on an event of probability at least $1-\e(-Cds^2)$, it holds 
\beq
\setlength{\jot}{5pt}
\begin{split}
\TV(\pi_v,\bvmgam) \les \frac{sd}{\sqrt{n\tmin} }+\frac{\bdelta_{01}(2s)}{\sqrt n}+e^{d(\Cpri^*-Cs^2)}.
\end{split} 
\eeq
\end{corollary}
\begin{corollary}[Asymptotic BvM for discrete probability distribution]\label{corr:bvmasympmf}Let $d=d_n$ and $\Theta_n$ be the corresponding sequence of parameter spaces~\eqref{Thetapmf}. Let $\theta^*_n\in\Theta_n$ be a sequence of ground truth pmfs, and write $v_n,\bvmgam_n,M_{0n}^*,\bdelta_{01n},\theta^*_{\mathrm{min},n}$ to emphasize the dependence of these quantities on $n$. If $M_{0n}^*=\mathcal O(1)$, $\bdelta_{01n}(s)=o(\sqrt n)$ for each $s\geq0$, and if
$$\frac{d_n^2}{n\theta^*_{\mathrm{min},n}}=o(1),$$ then $\TV(\pi_{v_n}, \bvmgam_n)\to0$ with probability tending to 1 as $n\to\infty$.
\end{corollary}
This result is stronger than the BvM result of~\cite{boucheron2009discrete} in two ways. First, we assume only that the prior satisfies $\bdelta_{01n}(s)=o(\sqrt n)$, whereas~\cite{boucheron2009discrete} essentially requires $\bdelta_{01n}(s)=o(1)$ provided one defines $\bdelta_{01n}(s)$ as a Lipschitz constant; recall the discussion at the end of Section~\ref{sec:bvmprelim}. Second, we require only that $d_n^2/n\theta^*_{\mathrm{min},n}=o(1)$, whereas~\cite{boucheron2009discrete} requires $d_n^3/n\theta^*_{\mathrm{min},n}=o(1)$.

\appendix

\section{Proofs from Section~\ref{gensec}}\label{app:gensec}
Recall that
\beq\label{app:del3redef}
\delta_3(r)=\sup_{|\theta-\MLE|_H\leq r\sqrt{d/n}}\frac{\|\nabla^2\f(\theta)-\nabla^2\f(\MLE)\|_H}{|\theta-\MLE|_H}, \quad H=\nabla^2f(\MLE),
\eeq where $\MLE$ is the unique minimizer of $\f$.
\subsection{Proofs from Section~\ref{subsec:outline}}\label{app:lma:TVprelim}
\begin{proof}[Proof of Lemma~\ref{lingrow}]To prove the lower bound on $\nabla^2\f(\theta)$ in $\UU(r)$, note that 
\beq\label{del3use1}\|\nabla^2\f(\theta)-H\|_H=\|\nabla^2\f(\theta)-\nabla^2\f(\MLE)\|_H\leq \delta_3(r)\|\theta-\MLE\|_H \leq r\sqrt{d/n}\delta_3(r)\leq 1/2,\eeq using the assumption on $\delta_3(r)$ from Theorem~\ref{corr:main}. Hence $\nabla^2\f(\theta)\succeq\frac12H$ for all $\theta\in\UU(r)$. To prove the linear growth bound, recall that $\f$ is convex by assumption. Fix a point $\theta\in\Theta$ such that $|\theta-\MLE|_H\geq r\sqrt{d/n}$. Then the whole segment between $\MLE$ and $\theta$ is also in $\Theta$. In particular, convexity of $\f$ gives
\beqsn
\f(\theta)-\f(\MLE)\geq \frac1t\l[\f(\MLE + t(\theta-\MLE))-\f(\MLE)\r],\quad t = \frac{r\sqrt{d/n}}{|\theta-\MLE|_H}.
\eeqsn
Dividing both sides by $|\theta-\MLE|_H$ , we get
\beqs\label{infpsir}
\frac{\f(\theta)-\f(\MLE)}{|\theta-\MLE|_H} \geq  \frac{\f(\MLE + t(\theta-\MLE))-\f(\MLE)}{|t(\theta-\MLE)|_H}
\geq \inf_{|u|_H=r\sqrt{d/n}}\frac{\f(\MLE+u)-\f(\MLE)}{|u|_H}=:\psi.
\eeqs To get the second inequality, we replaced $t(\theta-\MLE)$, which has $H$-norm $r\sqrt{d/n}$ by construction, with any $u$ such that $|u|_H=r\sqrt{d/n}$. To finish the proof, we bound $\psi$ from below. A Taylor expansion around $\theta=\MLE$ gives that for $|u|_H=r\sqrt{d/n}$, we have
\beqs
\f(\MLE+u)-\f(\MLE)&= \frac12\la\nabla^2\f(\MLE+\xi), u^{\otimes 2}\ra\geq \frac{1}4|u|_H^2,
\eeqs using that $\nabla^2\f(\theta)\succeq\frac12H$ for all $\theta\in\UU(r)$ to get the second inequality. We now divide by $|u|_H$ to get
\beqs\label{psir}
\frac{\f(\theta)-\f(\MLE)}{|u|_H}&\geq \frac{1}4|u|_H= \frac{r}{4}\sqrt{d/n}.
\eeqs  Substituting this bound into~\eqref{infpsir} concludes the proof.
\end{proof}
\begin{proof}[Proof of Lemma~\ref{lma:TVprelim}] We let $\UU^c$ be the complement of $\UU$ in $\R^d$, and omit the range of integration when the integral is over all of $\R^d$. Suppose $\mu\propto F$ and $\nnu\propto\hat F$. Then
\beqs\label{efg1}
2\TV(\mu,\nnu) = \int\l|\frac{F}{\int F}-\frac{\hat F}{\int \hat F}\r|\leq \frac{\int_{\UU^c}F}{\int F} + \frac{\int_{\UU^c}\hat F}{\int \hat F} + \int_\UU\l|\frac{F}{\int F}-\frac{\hat F}{\int \hat F}\r|.
\eeqs
Next, note that
\beqs\label{efg2}
\int_\UU&\l|\frac{F}{\int F}-\frac{\hat F}{\int \hat F}\r| - \int_\UU\l|\frac{F}{\int_\UU F}-\frac{\hat F}{\int_\UU \hat F}\r|\leq \int_\UU\l|\frac{F}{\int F}-\frac{F}{\int_\UU F}\r| + \int_\UU\l|\frac{\hat F}{\int_\UU \hat F}-\frac{\hat F}{\int \hat F}\r|\\
&=\int_\UU F\l|\l(\int_\UU F\r)^{-1}-\l(\int F\r)^{-1}\r|+\int_\UU \hat F\l|\l(\int_\UU \hat F\r)^{-1}-\l(\int \hat F\r)^{-1}\r|\\
&=\l|1-\frac{\int_\UU F}{\int F}\r| + \l|1-\frac{\int_\UU \hat F}{\int \hat F}\r|=\frac{\int_{\UU^c}F}{\int F} + \frac{\int_{\UU^c} \hat F}{\int \hat F}.
\eeqs Substituting~\eqref{efg2} into~\eqref{efg1} gives
\beqs\label{efg3}
2\TV(\mu,\nnu) &\leq 2\frac{\int_{\UU^c}F}{\int F} + 2\frac{\int_{\UU^c} \hat F}{\int \hat F} + \int_\UU\l|\frac{F}{\int_\UU F}-\frac{\hat F}{\int_\UU \hat F}\r|\\
& = 2\frac{\int_{\UU^c}F}{\int_\UU F} + 2\frac{\int_{\UU^c} \hat F}{\int_\UU \hat F} +2\TV\l(\mu\vert_\UU, \nnu\vert_\UU\r).\eeqs Dividing by 2 finishes the proof of the first statement. 

Next we prove the second statement. For brevity, redefine $\mu$ and $\nnu$ to be their restrictions to $\UU$ and recall that $\mu\propto e^{-n\f}$, $\nnu\propto e^{-n\g}$ on $\UU$. Let $T(x)=(nH)^{1/2}x$. Then 
$$(T_{\#}\mu)(y)\propto e^{-\f_H(y)},\quad (T_{\#}\nnu)(y)\propto e^{-\g_H(y)},\quad y\in (nH)^{1/2}\UU,$$ where 
$$\f_H(y)=n\f((nH)^{-1/2}y),\quad \g_H(y)=n\g((nH)^{-1/2}).$$ Note that $\nabla^2\f_H(y) =H^{-1/2}\nabla^2\f((nH)^{-1/2}y)H^{-1/2}\succeq\lambda I_d$ for all $y\in (nH)^{1/2}\UU$ by the assumption $\nabla^2f(x)\succeq \lambda H$ for all $x\in\UU$. Therefore, $T_{\#}\mu$ is $\lambda$-strongly concave, so it satisfies a log-Sobolev inequality (LSI) with constant $1/\lambda$~\cite{bakry2014analysis}. Using the affine invariance of TV distance, then Pinsker's inequality, then the LSI, we get
\beq\label{TVTrho}
\TV(\nnu,\mu)^2=\TV(T_{\#}\nnu, T_{\#}\mu)^2\leq\frac12\KL{T_{\#}\nnu}{T_{\#}\mu}\leq \frac1{4\lambda}\FI{T_{\#}\nnu}{T_{\#}\mu}
\eeq where
$$
\FI{T_{\#}\nnu}{T_{\#}\mu}=\E_{Y\sim T_{\#}\nnu}\l[\l|\nabla\log\frac{T_{\#}\nnu}{T_{\#}\mu}(Y)\r|^2\r].
$$ Now, $\log((T_{\#}\nnu/T_{\#}\mu)(Y)) = (\f_H-\g_H)(Y) = n(\f-\g)((nH)^{-1/2}Y).$ Therefore, 
\beqsn
\nabla\log((T_{\#}\nnu/T_{\#}\mu)(Y)) &= \sqrt nH^{-1/2}\l(\nabla (\f-\g)((nH)^{-1/2}Y)\r)\\
&\stackrel{d}{=} \sqrt nH^{-1/2}\nabla (\f-\g)(X),
\eeqsn where $X\sim\nnu$. Therefore,
$$
\FI{T_{\#}\nnu}{T_{\#}\mu} = n\E_{X\sim\nnu}\l[\l\|\nabla (\f-\g)(X)\r\|_H^2\r].
$$ Substituting this into~\eqref{TVTrho} finishes the proof.
\end{proof}
\subsection{Proof of Theorem~\ref{corr:main}}\label{app:prop:main}
In this section, we make the assumptions from Theorem~\ref{corr:main}. Recall that it suffices to bound the local expectation and the two tail integrals in~\eqref{TVFI-nll}. These quantities are bounded in Lemma~\ref{lma:loc} and Corollary~\ref{corr:tail} below. The proof of Proposition~\ref{corr:main} follows immediately by adding up these bounds.
\subsubsection{Local Expectation}
\begin{lemma}\label{lma:loc}
It holds
\beq\sqrt{\frac n2}\E_{\theta\sim\gamma_\f\vert\UU}\l[\|\nabla\f(\theta)-H(\theta-\MLE)\|_{H}^2\r]^{\frac12}\leq \frac{\delta_3(r)}{\sqrt2}\frac{d}{\sqrt n}.\eeq
\end{lemma}
\begin{proof}First note 
$$\nabla\f(\theta)-H(\theta-\MLE)=\int_0^1\l(\nabla^2\f(\MLE+t(\theta-\MLE))-\nabla^2\f(\MLE)\r)(\theta-\MLE)dt.$$ Hence if $\theta\in\UU(r)$ then
$$\|\nabla\f(\theta)-H(\theta-\MLE)\|_{H}\leq \int_0^1\delta_3(r)t\|\theta-\MLE\|_H^2dt =\frac12\delta_3(r)\|\theta-\MLE\|_H^2.$$
Therefore, using that $\MLE+(nH)^{-1/2}Z\sim\gamma_\f$ when $Z\sim\mathcal N(0, I_d)$, we get
\beqs\label{TVinter}
\E_{\theta\sim\gamma_\f\vert\UU}&\l[\|\nabla\f(\theta)-H(\theta-\MLE)\|_{H}^2\r]\leq\frac14\delta_3(r)^2\E_{\theta\sim\gamma_\f\vert\UU}[\|\theta-\MLE\|_H^4]\\
&=\frac{\delta_3(r)^2}{4n^2}\PP(|Z|\leq r\sqrt d)^{-1}\E\l[|Z|^4\mathds{1}(|Z|\leq r\sqrt d)\r] \leq\frac{3\delta_3(r)^2d^2}{4n^2}\PP(|Z|\leq r\sqrt d)^{-1}.
\eeqs
Now note that $$\PP(|Z|\leq r\sqrt d)\geq 1-e^{-d(r-1)^2/2}\geq 1-e^{-25/2}\geq 3/4$$ for all $d\geq1$ since $r\geq6$. Substituting this bound into~\eqref{TVinter} gives
\beq\label{TVU2}
\E_{\theta\sim\gamma_\f\vert\UU}\l[\|\nabla\f(\theta)-H(\theta-\MLE)\|_{H}^2\r] \leq \delta_3(r)^2\frac{d^2}{ n^2}.
\eeq
Taking the square root and then multiplying by $\sqrt{n/2}$ gives the desired bound.
\end{proof}
\subsubsection{Tail integrals}
\begin{lemma}\label{lma:intUc}Let $\UU=\UU(r)$ for some $r>2$. It holds 
\beq
\int_{\UU^c\cap\Theta}e^{n(\f(\MLE)-\f(\theta))}d\theta\leq (2\pi)^{d/2}n^{-d/2}\e\l(\l[2+\log r - r^2/4\r]d\r).
\eeq
\end{lemma}
\begin{proof}\eqref{nllgrowth} gives
\beqs\label{Ucap}
\int_{\UU^c\cap\Theta}e^{n(\f(\MLE)-\f(\theta))}d\theta&\leq \int_{|\theta-\MLE|_H\geq r\sqrt{d/n}}\e\l(-\frac r4\sqrt{nd}|\theta-\MLE|_H\r)d\theta\\
&=n^{-d/2}\int_{|\theta|\geq r\sqrt d}\e\l(-\frac r4\sqrt{d}|\theta|\r)d\theta.
\eeqs
Using Lemma~\ref{aux:gamma} with $a=r$ and $b=r/4$ (the lemma applies since we assumed $r^2/4>1$), we get
\beq\label{Ibd}
\int_{|\theta|\geq r\sqrt d}\e\l(-\frac r4\sqrt{d}|\theta|\r)d\theta\leq(2\pi)^{d/2}d\e\l(\l[\frac32+\log r - \frac{r^2}{4}\r]d\r).
\eeq To conclude, bound $de^{3d/2}$ by $e^{2d}$, noting that $de^{-d/2}\leq1$ for all $d$.
\end{proof}
\begin{lemma}\label{lma:intU}Let $\UU=\UU(r)$ for some $r\geq2$. Then
\beq
\int_{\UU}e^{n(\f(\MLE)-\f(\theta))}d\theta\geq \frac12(2\pi)^{d/2}(3n/2)^{-d/2}.\eeq 
\end{lemma}
\begin{proof}For all $|\theta-\MLE|_H\leq r\sqrt{d/n}$, it holds
\beqs\label{nllno}
\f(\theta)-\f(\MLE) &= \frac12\lla\nabla^2\f(\xi), (\theta-\MLE)^{\otimes 2}\rra\leq \frac34|\theta-\MLE|_H^2
\eeqs for a point $\xi$ on the interval between $\theta$ and $\MLE$. Here, we have used that $\sup_{\theta\in\UU}\|\nabla^2\f(\theta)\|_H\leq 3/2$, since $\|\nabla^2\f(\theta)-H\|_H\leq1/2$, shown in the proof of Lemma~\ref{lingrow}. Hence
\beqs\label{intUL}
\int_{\UU}e^{n(\f(\MLE)-\f(\theta))}d\theta&\geq \int_\UU\e\l(-\frac{3n}4\l|\theta-\MLE\r|_H^2\r)d\theta\\
& = (3n/2)^{-d/2}\int_{|u|\leq r\sqrt{3/2}\sqrt d}\e\l(-|u|^2/2\r)du\\
&= (3n/2)^{-d/2}(2\pi)^{d/2}\PP(|Z|\leq r\sqrt{3/2}\sqrt d).\eeqs 

 To conclude the proof note that when $r\geq2$ we have $\PP(|Z|\leq r\sqrt{3/2}\sqrt d)\geq \PP\l(|Z|\leq \sqrt6\sqrt d\r)\geq 1-e^{-d(\sqrt6-1)^2/2}\geq 1/2$ for all $d\geq1$.
\end{proof}
\begin{corollary}\label{lma:intUUc}
Let $\UU=\UU(r)$. If $r\geq6$ then
$$
\frac{\int_{\Theta\setminus\UU}e^{-n\f(\theta)}d\theta}{\int_{\UU}e^{-n\f(\theta)}d\theta} \leq 2e^{-5dr^2/36}\leq 2e^{-dr^2/9}.
$$ 
\end{corollary}
\begin{proof}
Dividing the upper bound on the numerator by the lower bound on the denominator gives
$$
\frac{\int_{\Theta\setminus\UU}e^{-n\f(\theta)}d\theta}{\int_{\UU}e^{-n\f(\theta)}d\theta} \leq 2\e\l(d\l[2+\log(\sqrt{3/2})+\log r - \frac{r^2}{4}\r]\r).
$$
We now note that $2+\log(\sqrt{3/2})\leq 2.21$, and that 
\beq\label{FML03}
2.21+\log r\leq\frac{r^2}{9}\quad\forall r\geq6.
\eeq Substituting this bound into the exponential concludes the proof.
\end{proof}
\begin{corollary}\label{corr:tail}
Let $r\geq6$. Then 
$$T:= \frac{\int_{\Theta\setminus\UU}e^{-n\f}}{\int_\UU e^{-n\f}} + \gamma_\f(\UU^c)\leq 3\e\l(-dr^2/9\r).$$
\end{corollary}
\begin{proof}
From the above corollary and a standard Gaussian tail bound (i.e. using that $\gamma_\f(\UU(r)^c)=\PP(|Z|\geq r\sqrt d)\leq \e(-d(r-1)^2/2)$) we get
\beqs\label{T1}
T \leq 2e^{-dr^2/9}+ \e\l(-\frac{d}2\l(r-1\r)^2\r)\leq 3e^{-dr^2/9},
\eeqs
since $(d/2)(r-1)^2\geq (d/2)(r/2)^2=dr^2/8$. 
\end{proof}

\section{Proofs from Section~\ref{sec:bvmprelim}}\label{app:towardsbvm}
\begin{defn}\label{def:Udel3delp}On the event $\{\exists!\,\mathrm{MLE}\,\MLE\}$, define $H=\nabla^2\nll(\MLE)$ and $\UU(r)=\{|\theta-\MLE|_H\leq r\sqrt{d/n}\}$. Assuming $\UU(r)\subset\Theta$, define 
\beq
\delta_3(r)=\sup_{\theta\in\UU(r)}\frac{\|\nabla^2\nll(\theta)-\nabla^2\nll(\MLE)\|_H}{|\theta-\MLE|_H}.
\eeq\end{defn} 
Next, recall the definition of the event $\bar E(s,\epsilon_2)$:
$$\bar E(s,\epsilon_2)=\{\exists!\,\mathrm{MLE}\,\MLE,\, |\MLE-\ground|_\F\leq s\sqrt{d/n}\}\cap E_2(\epsilon_2)\cap E_3(2s).$$
\begin{lemma}\label{lma:UtoUdet}
Suppose~\eqref{bUcond} holds. Then on the event $\bar E(s,\epsilon_2)$, it holds
$$\nabla^2\nll(\MLE)\succeq \frac14\F,$$ 
\beq\label{taus0}
\begin{gathered}
\UU(s/2)\subseteq\bU(2s),\qquad\delta_3(s/2)\leq8\bdelta_3(2s),\end{gathered}
\eeq  \beq\label{taus}(s/2)\delta_3(s/2)\sqrt{d/n}\leq 1/2.\eeq
\end{lemma}
See Section~\ref{app:towardsbvm:aux} for the proof.
\begin{proof}[Proof of Lemma~\ref{prop:lapMLE}]We split up the proof into steps. First note that by the assumption $\epsilon_2\leq1/2$, and recalling that $\F=\nabla^2\nll(\ground)$, we have $\nabla^2\nll(\ground)\succeq \frac12 \F$ on $E_2(\epsilon_2)$. Also, recall that $\nabla\nllinfty(\ground)=0$.
\begin{proof}[Proof that unique MLE $\MLE$ exists and satisfies $|\MLE-\ground|_\F\leq s\sqrt{d/n}$.]
We check the assumptions of Corollary~\ref{corr:lang2} with $h=\nll$, $y_0=\ground$, $q=s\sqrt{d/n}$, and $\lambda=1/2$. The condition $\nabla^2\nll(\ground)\succeq \lambda \F$ and $\|\nabla\nll(\ground)\|_\F\leq 2\lambda  q$ are both satisfied on $E(s,\epsilon_2)$. The condition $\|\nabla^2\nll(\theta)-\nabla^2\nll(\ground)\|_\F\leq\lambda/4$ for all $|\theta-\ground|_\F\leq q$ is also satisfied on this event, since
\beqs\label{nlltg}
\|\nabla^2\nll(\theta)-\nabla^2\nll(\ground)\|_\F&\leq s\bdelta_3(s)\sqrt{d/n}\leq1/8=\lambda/4.\eeqs Therefore, the assumptions of Corollary~\ref{corr:lang2} are satisfied so we conclude there exists $\MLE$ such that $\nabla\nll(\MLE)=0$ and $|\MLE-\ground|_\F\leq q= s\sqrt{d/n}$. Also, we immediately get by~\eqref{nlltg} that $\|\nabla^2\nll(\MLE) -\nabla^2\nll(\ground)\|_\F \leq 1/8$. Using this and the fact that $\nabla^2\nll(\ground)\succeq \frac12 \F$ we get $\nabla^2\nll(\MLE) \succeq \frac38\F$. Hence $\MLE$ is a strict local minimizer, and therefore a global minimizer, since $\nll$ is convex.\end{proof}
This first part of the proof shows that $E(s,\epsilon_2)\subset\bar E(s,\epsilon_2)$. In the rest of the proof, we show~\eqref{pretv1}-\eqref{pretv2} hold on $\bar E(s,\epsilon_2)$.
\begin{proof}[Proof of Bound~\eqref{pretv1}]By Assumption~\ref{A1}, $\nll$ is convex and $\nll\in C^2(\Theta)$ with probability 1. Also, $\nll$ has a unique global minimizer $\MLE$ on $\bar E(s,\epsilon_2)$. Let $r=s/2$, so that $r\geq6$. By Lemma~\ref{lma:UtoUdet}, it holds on $\bar E(s,\epsilon_2)$ that $\UU(r)\subset\bU(2s)\subset\Theta$ and $r\delta_3(r)\sqrt{d/n}\leq 1/2$. Therefore, we can apply Theorem~\ref{corr:main} with $f=\nll$ to get that $\TV(\pi_\nll,\gamma_\nll)\leq \delta_3(r)\frac{d}{\sqrt {2n}}+ 3e^{-dr^2/9}$. Using that $\delta_3(r)\leq 8\bdelta_3(2s)$ on $\bar E(s,\epsilon_2)$ and that $e^{-dr^2/9}=e^{-ds^2/36}$ concludes the proof.\end{proof}

\begin{proof}[Proof of Bound~\eqref{pretv2}]Let $r= s/2$, so that $r\geq6$. Since $2s\bdelta_3(2s)\sqrt{d/n}\leq1/4$ and $\nabla^2\nll(\ground)\succeq\frac12\F$ on $\bar E(s,\epsilon_2)$, it follows that $\nabla^2\nll(\theta)\succeq\frac14\F$ for all $\theta\in\bU(2s)$. We apply Lemma~\ref{lma:TVprelim} with $\UU=\bU(2s)$, $\mu=\pi_\nll$, $\nnu=\pi_v$, $\lambda=1/4$, and $\F$ instead of $H$. This gives
\beq
\TV(\pi_v,\pi_\nll)\leq \frac{\int_{\Theta\setminus\bU(2s)}e^{-n\nll}\prior}{\int_{\bU(2s)}e^{-n\nll}\prior} + \frac{\int_{\Theta\setminus\bU(2s)}e^{-n\nll}}{\int_{\bU(2s)}e^{-n\nll}} + \sqrt n\E_{\theta\sim\pi_v\vert_\UU}\l[\|\nabla(n^{-1}\nlp)\|_\F^2\r]^{\frac12}.\eeq In the last term, we simply bound the square root of the expectation of the square by the maximum of $\|\nabla(n^{-1}\nlp(\theta))\|_\F$ over $\theta\in\bU(2s)$, which is $n^{-1}\bdelta_{01}(2s)$ by definition. Thus
\beq\label{FML0}
\sqrt n\E_{\theta\sim\nu\vert_\UU}\l[\|\nabla(n^{-1}\nlp)\|_\F^2\r]^{\frac12}\leq\frac{1}{\sqrt n}\bdelta_{01}(2s).
\eeq

Now, by definition of $\Cpri^*$ and $\bdelta_{01}$ and using that $\UU(r)\subset\bU(2s) $ on $\bar E(s,\epsilon_2)$ by Lemma~\ref{lma:UtoUdet}, we have
\beqs\label{Ucprior0}
\frac{\int_{\Theta\setminus\bU(2s)}e^{-n\nll}\prior}{\int_{\bU(2s)} e^{-n\nll}\prior} &\leq \frac{\sup_{\theta\in\Theta}\prior(\theta)/\prior(\ground)}{\inf_{\theta\in\bU(2s)}\prior(\theta)/\prior(\ground)}\frac{\int_{\Theta\setminus\bU(2s)} e^{-n\nll}}{\int_{\bU(2s)} e^{-n\nll}}\\
&\leq \e(d\Cpri^*+2s\bdelta_{01}(2s)\sqrt{d/n})\frac{\int_{\Theta\setminus\bU(2s)} e^{-n\nll}}{\int_{\bU(2s)} e^{-n\nll}} \\
&\leq \e(d\Cpri^*+ds/3)\frac{\int_{\Theta\setminus\UU(r)} e^{-n\nll}}{\int_{\UU(r)} e^{-n\nll}}.
\eeqs
In the last line we used that $\bdelta_{01}(2s)\leq \sqrt{nd}/6$. Hence
\beqs\label{tailFML}
 \frac{\int_{\Theta\setminus\bU(2s)}e^{-n\nll}\prior}{\int_{\bU(2s)}e^{-n\nll}\prior} &+ \frac{\int_{\Theta\setminus\bU(2s)}e^{-n\nll}}{\int_{\bU(2s)}e^{-n\nll}}\leq 2 \e\l(d[\Cpri^* + s/3]\r)\frac{\int_{\Theta\setminus\UU(r)} e^{-n\nll}}{\int_{\UU(r)} e^{-n\nll}}\\
\leq 4\e\bigg(d&\l[\Cpri^*+s/3 - 5r^2/36\r]\bigg)\leq 4\e\l(d(\Cpri^*-s^2/144)\r).
\eeqs We used~\eqref{Ucprior0} to get the first inequality and Corollary~\ref{lma:intUUc} to get the second inequality. To get the third inequality we used that $5r^2/36=5s^2/144$ and that $s/3 \leq 4s^2/144$ when $s\geq12$. Combining~\eqref{tailFML} with~\eqref{FML0} concludes the proof.


%
%
\end{proof}
\begin{proof}[Proof of Bound~\eqref{pretv3}]
We use Lemma~\ref{lma:TVgauss} with $\Sigma_1^{-1}=n\nabla^2\nllinfty(\ground)=n\F$ and $\Sigma_2^{-1}=n\nabla^2\nll(\MLE)$. We have $\nabla^2\nll(\MLE)\succeq \frac14\F$, so we can take $\tau=1/4$. Also, we have 
\beqsn
\|n\nabla^2\nll(\MLE)-n\F\|_{n\F}&=\|\nabla^2\nll(\MLE)-\F\|_{\F}\leq \|\nabla^2\nll(\MLE)-\nabla^2\nll(\ground)\|_\F + \|\nabla^2(\nll-\nllinfty)(\ground)\|_\F \\
&\leq s\bdelta_3(s)\sqrt{d/n} +\epsilon_2 =:\epsilon
\eeqsn on $\bar E(s,\epsilon_2)$. Substituting the above values of $\tau$ and $\epsilon$ into~\eqref{TVD12} concludes the proof of~\eqref{pretv3}.
\end{proof}
\end{proof}

\begin{proof}[Proof of bounds in Example~\ref{ex:prior}]
The bounds for the flat prior are trivial. For the Gaussian prior $\prior=\mathcal N(\mu, \Sigma)$, we have $\log\prior(\theta)=-\frac12(\theta-\mu)^\tr \Sigma^{-1}(\theta-\mu)+\mathrm{const}$, so that
\beqs
d\Cpri^* =\sup_{\theta\in\R^d}\log\prior(\theta)-\log\prior(\ground) =  \frac12(\ground-\mu)^\tr \Sigma^{-1}(\ground-\mu)=\frac12|\ground-\mu|_{\Sigma^{-1}}^2 .
\eeqs
Next, we have $\nabla\log\prior(\theta)=\Sigma^{-1}(\theta-\mu)$, and therefore 
$$\|\nabla\log\prior(\theta)\|_\F=|\F^{-1/2}\Sigma^{-1}(\theta-\mu)| \leq\|\Sigma^{-1}\|_\F|\theta-\mu|_\F.$$ Hence
$$\bdelta_{01}(r)\leq\|\Sigma^{-1}\|_\F(|\mu-\ground|_\F+r\sqrt{d/n}).$$ For the multivariate student's t prior $\prior=t_\nu(\mu, \Sigma)$, we have $\log\prior(\theta)=-\frac{\nu+d}{2}\log\l(1+\frac1\nu|\theta-\mu|^2_{\Sigma^{-1}}\r)+\mathrm{const}$, so 
\beqs
\log\prior(\theta)&-\log\prior(\ground) = \frac{\nu+d}{2}\log\l(\frac{1+\frac1\nu|\ground-\mu|^2_{\Sigma^{-1}}}{1+\frac1\nu|\theta-\mu|^2_{\Sigma^{-1}}}\r)\\
&\leq  \frac{\nu+d}{2}\log\l(1+\frac1\nu|\ground-\mu|^2_{\Sigma^{-1}}\r)\leq\frac{\nu+d}{2\nu}|\ground-\mu|^2_{\Sigma^{-1}}.
\eeqs Hence $\Cpri^*\leq \frac{\nu+d}{2\nu d}|\ground-\mu|^2_{\Sigma^{-1}}$. Finally,
$$\nabla\log\prior(\theta) = -\frac{\nu+d}{\nu}\frac{\Sigma^{-1}(\theta-\mu)}{1+\frac1\nu|\theta-\mu|^2_{\Sigma^{-1}}},$$ so
\beqs
\bdelta_{01}(r)&=\sup_{\theta\in\UU(r)}\|\nabla\log\prior(\theta)\|_\F \leq \frac{\nu+d}{\nu}\sup_{\theta\in\UU(r)}\|\Sigma^{-1}(\theta-\mu)\|_{\F}\\
&\leq \frac{\nu+d}{\nu}\|\Sigma^{-1}\|_\F(|\mu-\ground|_\F+r\sqrt{d/n}).
\eeqs
\end{proof}

\subsection{Auxiliary results}\label{app:towardsbvm:aux}
\begin{proof}[Proof of Lemma~\ref{lma:UtoUdet}]Using~\eqref{bUcond}, the definition of event $\bar E(s,\epsilon_2)$, and Taylor's theorem, we have 
\beqsn
\|\nabla^2\nll(\MLE)-\nabla^2\nll(\ground)\|_\F &\leq \sup_{|\theta-\ground|_\F\leq s\sqrt{d/n}}\|\nabla^2\nll(\theta)-\nabla^2\nll(\ground)\|_\F\\
&\leq s\bdelta_3(s)\sqrt{d/n} \leq 1/4.\eeqsn
Since $\nabla^2\nll(\ground)\succeq\frac12\F$ on $\bar E(s,\epsilon_2)$ (using that $\epsilon_2\leq1/2$), we conclude that $H=\nabla^2\nll(\MLE)\succeq\frac14F$. This proves the first statement. Next, fix $\theta\in\UU(s/2)$, so that $|\theta-\MLE|_H\leq (s/2)\sqrt{d/n}$. Using that $H\succeq\frac14 \F$ and $|\MLE-\ground|_\F\leq s\sqrt{d/n}$ on $\bar E(s,\epsilon_2)$, we then have
$$
|\theta-\ground|_\F \leq |\theta-\MLE|_\F +|\MLE-\ground|_\F \leq 2|\theta-\MLE|_H + s\sqrt{d/n} \leq 2s\sqrt{d/n}.
$$ We conclude that $\UU(s/2)\subseteq \bU(2s)$. Next, using that $\UU(s/2)\subset\bU(2s)$ and that $\MLE\in\bU(s)\subset\bU(2s)$ on event $\bar E(s,\epsilon_2)$, we have
\beqsn
\delta_3(s/2)&=\sup_{\theta\in\UU(s/2)}\frac{\|\nabla^2\nll(\theta)-\nabla^2\nll(\MLE)\|_H}{|\theta-\MLE|_H} \leq \sup_{\theta,\theta'\in \bU(2s)}\frac{\|\nabla^2\nll(\theta)-\nabla^2\nll(\theta')\|_H}{|\theta-\theta'|_H} \\
&\leq 8\sup_{\theta,\theta'\in \bU(2s)}\frac{\|\nabla^2\nll(\theta)-\nabla^2\nll(\theta')\|_\F}{|\theta-\theta'|_\F}\leq 8\bdelta_3(2s).
\eeqsn Here, we used that $H\succeq\frac14\F$ on $\bar E(s,\epsilon_2)$.
Finally,
$$(s/2)\delta_3(s/2)\sqrt{d/n} \leq (s/2)(8\bdelta_3(2s))\sqrt{d/n} =4s\bdelta_3(2s)\sqrt{d/n}\leq 1/2,$$ proving~\eqref{taus}.
\end{proof}

\section{Proofs from Section~\ref{sec:log} (GLMs)}\label{app:sec:logreg}
\begin{proof}[Proof of Lemma~\ref{bern-gauss}]Let 
$$Q=\frac1n\sum_{i=1}^nX_i\l[Y_i-\nabla\psi(X_i^\tr \ground)\r],$$ treating $\nabla\psi$ as a column vector in $\R^k$. We are interested in bounding from above the tail probability $\PP(\|Q\|_{\F}\geq r\sqrt{d/n})$. Let $\mathcal N$ be a $1/2$-net of the sphere in $\R^d$, i.e. a collection of unit norm vectors $u$ such that $\inf_{u\in\mathcal N}|u-w|\leq 1/2$ for all $|w|=1$. By standard arguments, we can take $\mathcal N$ to have at most $5^d$ elements. Next, let $\mathcal N_{\F}:=\{{\F}^{-1/2}u\mid u\in\mathcal N\}$. Then
$$\|Q\|_\F\leq 2\sup_{u\in\mathcal N_\F}u^\tr Q,$$ and hence
\beq\label{pb1}\PP(\|Q\|_{\F}\geq s\sqrt{d/n})\leq \PP(\sup_{u\in\mathcal N_{\F}}u^\tr Q\geq (s/2)\sqrt{d/n})\leq 5^d\sup_{|u|_{\F}=1}\PP(u^\tr Q\geq (s/2)\sqrt{d/n}),\eeq by a union bound. Now, fix $|u|_{\F}=1$. Then
\beqs\label{pb2}
\PP&(u^\tr Q\geq (s/2)\sqrt{d/n}) =\PP\l(\sum_{i=1}^nu^\tr X_i(Y_i-\nabla\psi(X_i^\tr\ground))\geq (s/2)\sqrt{nd}\r)\\
&\leq\e\l(-\tau \frac s2\sqrt{nd}\r)\prod_{i=1}^n\E\l[\exp(\tau u^\tr X_i(Y_i-\nabla\psi(X_i^\tr\ground))\r]\\
&=\exp\l(-\tau \frac s2\sqrt{nd}+\sum_{i=1}^n\l[\psi(X_i^\tr(\ground + \tau u))-\psi(X_i^\tr\ground)-\tau u^\tr X_i\nabla\psi(X_i^\tr\ground)\r]\r)\\
&=\exp\l(-\tau \frac s2\sqrt{nd}+n\l[h_u(\tau)-h_u(0)-\tau h_u'(0)\r]\r),
\eeqs where
$$h_u(t) = \frac1n\sum_{i=1}^n\psi(X_i^\tr\ground+ tX_i^\tr u).$$ 
We used Chernoff's inequality and standard properties of exponential families in the second to last line, which is valid for any $\tau$ such that $X_i^\tr(\ground + \tau u)\in\Omega$ for all $i$. In particular, since $\bU(2s)\subset\Theta$ we know by the definition of $\Theta$ that $X_i^\tr(\ground+ \tau u)\in\Omega$ for all $i$ and all $\tau\leq 2s\sqrt{d/n}$, $|u|_\F=1$. We take $\tau=(s/2)\sqrt{d/n}$. This gives, for some $\xi\in[0,\tau]$,
\beqs\label{pb3}
\sup_{|u|_\F=1}\l\{h_u(\tau)-h_u(0)-\tau h_u'(0)\r\} &=\sup_{|u|_\F=1} \frac{\tau^2}{2}\l(h_u''(0)+\frac{\tau}{3}h_u'''(t\xi)\r)\\
&\leq \frac{\tau^2}{2}\l(1+\frac{\tau}{3}\sup_{|u|_\F=1,t\in[0,\tau]}h_u'''(t)\r)\leq  \frac{\tau^2}{2}\l(1+\frac{\tau}{3}\bdelta_3(s/2)\r).\eeqs To get the first inequality we used that $h_u''(0) = u^\tr\F u$. To get the second inequality, we used that
\beqsn
h_u'''(t) =\frac1n\sum_{i=1}^n\la\nabla^3\psi(X_i^\tr(\ground+ tu)), (X_i^\tr u)^{\otimes 3}\ra =\la\nabla^3\nll(\ground+tu), u^{\otimes3}\ra \leq \|\nabla^3\nll(\ground+tu)\|_{\F}\eeqsn for all $|u|_\F=1$ (recall the third line of~\eqref{nlldef}). Using~\eqref{pb3} and that $\tau=(s/2)\sqrt{d/n}$, we obtain
\beqs
5^d\PP(u^\tr Q\geq (s/2)\sqrt{d/n})&\leq \exp\l(d\log5-\tau \frac s2\sqrt{nd}+\frac{n\tau^2}{2}\l(1+\frac{\tau}{3}\bdelta_3(s/2)\r)\r)\\
&= \exp\l(d\log5- \frac{s^2d}{4}+\frac{s^2d}{8}\l(1+\frac{s}{6}\bdelta_3(s/2)\sqrt{d/n}\r)\r)\\
&\leq \exp\l(ds^2\l[\frac{\log5}{12^2}-\frac{1}{4}+\frac{1}{8}(1+1/48)\r]\r)\leq e^{-ds^2/10}.
\eeqs To get the first inequality in the last line we used that $2s\bdelta_3(2s)\sqrt{d/n}\leq1/4$ and $s\geq12$. 
\end{proof} 
\begin{proof}[Proof of Lemma~\ref{lma:adam}]The following calculation has been shown in~\cite{ak-expand}; we repeat it here for convenience. 
Let $m=\max(d,\sqrt{\log n})$ which ensures $d\leq m\leq n\leq e^{\sqrt m}$. This is needed to apply the result of~\cite{adamczak2010quantitative} cited below. Let $\bar X_i=(X_i, \tilde X_i)$, where $\tilde X_i\iid\mathcal N(0, I_{m-d})$, and the $\tilde X_i$ are independent of the $X_i$. Then $\bar X_i\iid\mathcal N(0, I_m)$ and we have
\beq\label{Xiu}
\sup_{u\in S^{d-1}}\frac1n\sum_{i=1}^n|X_i^\tr u|^k \leq \sup_{u\in S^{m-1}}\frac1n\sum_{i=1}^n|\bar X_i^\tr u|^k
\eeq with probability 1. Here, $S^{d-1}$ is the unit sphere in $\R^d$. Next, by~\cite[Proposition 4.4]{adamczak2010quantitative} applied with $t=s=1$, it holds
\beqs\label{ub-c3}
 \sup_{u\in S^{m-1}}\frac1n\sum_{i=1}^n|\bar X_i^\tr u|^k \leq  &\sup_{u\in S^{m-1}}\frac1n\sum_{i=1}^n\E\l[|\bar X_i^\tr u|^k\r]\\
 &+C_k\log^{k-1}\l(\frac{2n}{m}\r)\sqrt{\frac mn} + C_k\frac{m^{k/2}}{n} + C_k\frac{m}{n}
\eeqs with probability at least
\beq\label{prob-c3}
1-\exp(-C\sqrt n)-\exp(-C_k\min(n\log^{2k-2}(2n/m), \sqrt{nm}/\log(2n/m)))\geq 1-2\exp(-C_k\sqrt n),
\eeq where $C$ is an absolute constant and $C_k$ depends only on $k$. The inequality in~\eqref{prob-c3} follows by noting that $\log(2n/m)\geq\log 2$ and $\log(2n/m)\leq\log(2e^{\sqrt m})\leq\log(e^{2\sqrt m})=2\sqrt m$, and therefore, $n\log^{2k-2}(2n/m)\geq C_kn$ and $\sqrt{nm}/\log(2n/m)\geq \sqrt n/2$.
We can also further upper bound~\eqref{ub-c3} by using that $m\leq n$ and $\E\l[|\bar X_i^\tr u|^k\r]=C_k$ for all $i$ and some $C_k$. Therefore,
\beq\label{further-ub}
 \sup_{u\in S^{m-1}}\frac1n\sum_{i=1}^n|\bar X_i^\tr u|^k \leq C_k\l(1+\frac{m^{k/2}}{n}\r) \leq C_k\l(1+\frac{d^{k/2}}{n}\r),
\eeq where $C_k$ may change value. The second inequality uses that $m\leq d+\sqrt{\log n}$. Combining~\eqref{Xiu} with~\eqref{further-ub} completes the proof.
\end{proof}
\subsection{Analysis of BvM in~\cite{belloni2014posterior} and~\cite{spokoiny2013bernstein}}\label{subsec:bell-spok}
In this section, all equation and theorem numbers refer to the works of~\cite{belloni2014posterior} and~\cite{spokoiny2013bernstein} unless otherwise specified with the signifier ``our".

We study the proof of Theorem 2.2 of~\cite{belloni2014posterior} for $\alpha=0$; this recovers the TV distance. The main preliminary bound is (B.1). The first term on the righthand side has to do with the prior. For simplicity we neglect it; i.e. assume a flat prior. The third and fourth terms on the righthand side are exponentially small tail integrals. Thus the main contribution stems from the second term, which is bounded in the four displays starting with (B.2). The fourth of these displays, with $\alpha=0$, gives the result:
\beq\label{TVbell}\TV(\pi_v,\bvmgam)=\mathcal O(r_{1n}+e^{-r_{2n}\bar cM_{d,0}}+r_{3n})\eeq with high probability. Here $M_{d,0}=d$; this quantity is defined above Theorem 2.2. The quantities $r_{1n},r_{2n},r_{3n}$ are defined in Lemma A.5. Since $r_{2n}$ is bounded below by a constant; thus the middle term is exponentially small in $d$. When the prior is flat, then $r_{1n}=0$. Finally, $r_{3n}=cd\lambda_n(c)$ for an absolute constant $c$. Using this in our~\eqref{TVbell} gives
\beq\label{r3n}\TV(\pi_v,\bvmgam)=\mathcal O(d\lambda_n(c)).\eeq Here, $\lambda_n(c)$ is defined in (2.6), in terms of $B_{1n}(0)$ and $B_{2n}(c)$ defined in (2.4) and (2.5). Note that $B_{1n}(c)$ from (2.4) is precisely equal to $\bdelta_3(\sqrt c)$ in our notation; this follows from the representation in our~\eqref{F-bdel3-expfam-2}. To continue in the style of our notation, we can write $B_{2n}(c)=\bdelta_4(\sqrt c)$, where
\beq\label{F-bdel4-expfam-2}
\bdelta_4(s)= \sup_{\theta\in\bU(s),\,u\neq0}\;\frac{\E_\theta\l[\l(u^\tr Y-u^\tr \E_\theta[Y]\r)^4\r]}{\Var_\ground(u^\tr Y)^{2}}
\eeq Now, in our notation, we have
$$\lambda_n(s^2)=\frac s6\sqrt{d/n}\l(\bdelta_3(0)+s\sqrt{d/n}\bdelta_4(s)\r).$$
Substituting this into our~\eqref{r3n} and writing $c=s^2$ finally gives
$$
\TV(\pi_v,\bvmgam)=\mathcal O\l(\l(\bdelta_3(0)+s\sqrt{d/n}\bdelta_4(s)\r)\frac{d\sqrt d}{\sqrt n}\r).
$$
To interpret the model-dependent factor, note that we can write $\bdelta_{3,4}$ as follows: $\bdelta_k(s)= \sup_{\theta\in\bU(s)}\|\nabla^k\psi(\theta)\|_{\F}$ for $k=3,4$. Thus a Taylor expansion shows that $\bdelta_3(0)+\bdelta_4(s)s\sqrt{d/n}\geq \bdelta_3(s)$.\\

Next, we apply the result of~\cite{spokoiny2013bernstein} to GLMs. In condition $(ED_2)$, the quantity $\nabla^2\xi(\theta)\equiv0$ due to stochastic linearity. Therefore, $\omega$ can be taken arbitrarily small. Thus we set $\omega=0$.  Substituting $\omega=0$ into the quantity $\Diamond(r_0,x)$ defined in (2.6) of Theorem 2.2 gives $\Diamond(r_0,x)=r_0\delta(r_0)$. Thus, the last two inequalities in Theorem 2.4 roughly give that \beq\label{TVspok}\TV(\pi_v,\bvmgam)\les \Delta_0(x):=r_0\Diamond(r_0,x)=r_0^2\delta(r_0)\eeq with probability $1-5e^{-x}$ for a suitably chosen $r_0$ depending on $x$. Roughly speaking, $r_0^2=C(d+x)$; see the discussion in Section 2.4. Thus for simplicity, and to parallel our results, we take $r_0=s_0\sqrt d$.

Now, the quantity $\delta(r_0)$ is defined in $(\mathcal L_0)$. Using that $r_0=s_0\sqrt d$, the condition defining $\delta(r_0)$ is as follows, in our notation:
$$\sup_{\theta\in\bU(s_0)}\|\nabla^2\nllinfty(\theta)-\nabla^2\nllinfty(\ground)\|_{\F} \leq \delta(r_0).$$ Thus we can write $\delta(r_0)=\tilde\delta^*_3(s_0)s_0\sqrt{d/n},$ where $\tilde\delta^*_3$ is comparable to our $\bdelta_3$. Substituting this into our~\eqref{TVspok} gives
$$\TV(\pi_v,\bvmgam)\les r_0^2\delta(r_0)=s_0d\tilde\delta^*_3(s_0)s_0\sqrt{d/n}=s_0^2\tilde\delta^*_3(s_0)\frac{d\sqrt d}{\sqrt n}.$$

 \section{Proofs from Section~\ref{sec:pmf}}\label{app:sec:pmf}
\begin{proof}[Proof of Lemma~\ref{lma:thetabd}]To prove the $\chi^2$ characterization of $\bU(r)$, fix $\theta\in\R^d$ and let $\theta_0:=1-\theta^\tr \one$. Now note that
\beqs
|\theta-\ground|_\F^2 &= (\theta-\ground)^\tr \F(\theta-\ground) = \sum_{j=1}^d\frac{(\theta_j-\theta_j^*)^2}{\theta_j^*}+ \frac{((\theta-\ground)^\tr \one)^2}{\theta_0^*}\\
&= \sum_{j=1}^d\frac{(\theta_j-\theta_j^*)^2}{\theta_j^*}+ \frac{(\theta_0-\theta_0^*)^2}{\theta_0^*}=\chi^2(\theta||\ground).
\eeqs Hence $\UU(r)=\{\theta\in\R^d\; : \;\chi^2(\theta||\ground)\leq r^2d/n\}$. Note that in the first line above, $\ground$ denotes $(\theta_1^*,\dots,\theta_d^*)$. Next, suppose $\bU(2r)\not\subseteq\Theta$. We will show that $\tmin\leq 4r^2d/n$. If $\bU(2r)\not\subseteq\Theta$, then there exists $\theta\in\bU(2r)$ and $j\in\{1,\dots,d\}$ such that $\theta_j\leq0$ or $\theta_j\geq1$. Suppose first that $\theta_j\leq0$. Then
\beq
\theta_j^*\leq \frac{(\theta_j - \theta_j^*)^2}{\theta_j^*}\leq \chi^2(\theta||\ground)\leq 4r^2d/n\eeq and hence $\tmin\leq 4r^2d/n$. Next, suppose $\theta_j\geq1$. Then $\theta_0\leq 0$, and we can apply the above argument with $j=0$. Finally, fix $\theta\in\bU(r)\subset\Theta$. Then the bound~\eqref{max2chi} gives
$$
\max_{j=0,\dots,d}\l|\frac{\theta_j}{\theta_j^*}-1\r|^2\leq \frac{\chi^2(\theta||\ground)}{\tmin} \leq \frac{r^2d/n}{\tmin}\leq \frac14$$ by assumption on $r$, and hence $\theta_j\geq\theta_j^*/2\;\forall j=0,1,\dots,d$, as desired.
\end{proof}

Throughout the rest of the section, we use the shorthand notation
$$\chi^2 = \chi^2(\bar N||\ground).$$

\begin{proof}[Proof of (4)]Using (4.3) of~\cite{boucheron2009discrete}, we have
\beq\label{chileq}\chi\leq C\sqrt{d/n} + C\sqrt{x/n} + C\frac{x}{n\sqrt{\tmin}}\eeq with probability at least $1-e^{-x}$ for any $x>0$. Now, let $x=dt^2$ for some $t\geq1$ for which $\frac{t^2d}{n\tmin}\leq1$. Since $t\geq1$ we can upper bound~\eqref{chileq} as 
\beq\label{chileq2}\chi \leq C\sqrt{t^2d/n} + C\frac{t^2d}{n\sqrt\tmin} = C\sqrt{\frac{t^2d}{n}}\l(1 +\sqrt{\frac{t^2d}{n\tmin}}\r)\leq C_1\sqrt{\frac{t^2d}{n}},\eeq using the assumption $\frac{t^2d}{n\tmin}\leq1$. The bound~\eqref{chileq2} holds with probability at least $1-e^{-t^2d}$. We now choose $t= s/C_1$, which satisfies $t\geq1$ and $\frac{t^2d}{n\tmin}\leq1$ if $s\geq C_1$ and $s^2d/n\leq C_1^2\tmin$. Thus $E_0(s)=\{\chi\leq s\sqrt{d/n}\}$ holds with probability at least $1-e^{-t^2d} = 1-e^{-s^2d/C_1^2}$.
\end{proof}
\begin{remark}\label{rk:chibdd}
Since we have assumed $s^2d/n\leq C_1^2\tmin$, note that we also have $\chi^2/\tmin\leq s^2d/(n\tmin) \leq C_1^2$ on $E_0(s)$.\end{remark}
\begin{proof}[Proof of (3)]First recall that if $(2s)^2d/n\leq \tmin/4$, then Lemma~\ref{lma:thetabd} shows that 
\beq\label{recalltheta}
\frac{\theta_j^*}{\theta_j}\leq 2\quad\forall j=0,\dots,d\quad\forall \theta\in\bU(2s).\eeq 
Now, fix $\theta\in\bU(2s)$. We upper bound $\|\nabla^3\nll(\theta)\|_\F$, which is given by the maximum of
\beqs\label{obj}
\lla\nabla^3\nll(\theta), u^{\otimes3}\rra = -2\sum_{j=1}^d\frac{\bar N_j}{\theta_j^3}u_j^3 + 2\frac{\bar N_0}{\theta_0^3}(u^\tr \one)^3=-2\sum_{j=0}^d\frac{\bar N_j}{\theta_j^3}u_j^3
\eeqs over all $u$ such that
\beq\label{quadconst}
1 = u^\tr \F u =  \sum_{j=1}^d\frac{u_j^2}{\theta_j^*}+ \frac{(u^\tr \one)^2}{\theta_0^*}=\sum_{j=0}^d\frac{u_j^2}{\theta_j^*},
\eeq where $u_0:=-\sum_{j=1}^du_j$. Dropping this constraint on the sum of the $u_j$ and using Lagrange multipliers, we see that the maximum is achieved at a $u$ such that $u_j=0$ for $j\in I$ and $u_j = \lambda\theta_j^3/(\theta_j^*\bar N_j)$ for $j\in\{0,\dots,d\}\setminus I$, where $I$ is some non-empty index set. The constraint~\eqref{quadconst} reduces to $\lambda^2S=1$ and the objective~\eqref{obj} reduces to $-2\lambda^3S$, where $S=\sum_{i\in I}\frac{\theta_j^6}{(\theta_j^*)^3\bar N_j^2}$. Thus $\lambda=-1/\sqrt S$, and the objective~\eqref{obj} is given by $-2\lambda^3S=2/\sqrt S$, which is maximized when $S=\min_{j=0,\dots,d}\frac{\theta_j^6}{(\theta_j^*)^3\bar N_j^2}$. Thus
\beq\label{nablaA}
\|\nabla^3\nll(\theta)\|_\F \leq  2\max_{j=0,\dots,d}\frac{\bar N_j(\theta_j^*)^{3/2}}{\theta_j^3}.\eeq We now take the supremum over all $\theta\in\bU(2s)$ and apply~\eqref{recalltheta} to get that
\beq
\bdelta_3(2s)=\sup_{\theta\in\bU(2s)}\|\nabla^3\nll(\theta)\|_\F\leq C\max_{j=0,\dots,d}\frac{1}{\sqrt{\theta_j^*}}\frac{\bar N_j}{\theta_j^*}\leq \frac{C}{\sqrt\tmin}\l(1 +\frac{\chi}{\sqrt\tmin} \r).
\eeq The last inequality follows by~\eqref{max2chi} applied with $\theta=\bar N$. Finally, as noted in Remark~\ref{rk:chibdd}, we have that $\chi/\sqrt{\tmin}$ is bounded by an absolute constant on $E_0(s)$. This concludes the proof.
\end{proof}
\begin{proof}[Calculation from Remark~\ref{rk:pmin}]
In the case that $\bar N_j=\theta_j^*$, we have as in the above proof that
\beq
\|\nabla^3\nll(\ground)\|_\F = \sup\l\{-2\sum_{j=0}^d\frac{u_j^3}{{\theta_j^*}^2}\; : \; \sum_{j=0}^du_j=0, \sum_{j=0}^d\frac{u_j^2}{\theta_j^*}=1\r\}.
\eeq
Using Lagrange multipliers we get that $u_j=c_1\theta_j^*$ for all $j$ in some nonempty index set $I$, and $u_j=c_2\theta_j^*$ for all other $j$. If $P=\sum_{j\in I}\theta_j^*$, the constraints reduce to $c_1P+c_2(1-P)=0$ and $c_1^2P+c_2^2P=1$. Solving for $c_1,c_2$ we get $c_1=-\sqrt{(1-P)/P}$ and $c_2=\sqrt{P/(1-P)}$. The objective value is then equal to $2[(1-P)^{3/2}/\sqrt P - P^{3/2}/\sqrt{1-P}]  = 2(1-2P)/\sqrt{P(1-P)}$. This is maximized when $P$ is smallest, which is $P=\tmin$. This yields
\beq
\|\nabla^3\nll(\ground)\|_\F = 2\frac{1-2\tmin}{\sqrt{\tmin}\sqrt{1-\tmin}}.
\eeq
\end{proof}

\begin{proof}[Proof of (2)]
We have
\beqs\label{nablabarN2}
\|\nabla^2\nll(\ground) - \F\|_\F &= \sup_{u\neq0}\frac{\sum_{j=1}^d\frac{\bar N_j-\theta_j^*}{(\theta_j^*)^2}u_j^2 + \frac{\bar N_0-\theta_0^*}{(\theta_0^*)^2}(\one^\tr u)^2}{\sum_{j=1}^d\frac{1}{\theta_j^*}u_j^2 + \frac{1}{\theta_0^*}(\one^\tr u)^2}\\
&\leq \max_{j=0,1,\dots, d}\frac{|\bar N_j-\theta_j^*|}{\theta_j^*} \leq \frac{\chi}{\sqrt\tmin} \leq \sqrt{\frac{s^2d}{n\tmin }},
\eeqs using~\eqref{max2chi} to get the second-to-last inequality, and that $\chi\leq s\sqrt{d/n}$ on $E_0(s)$ to get the last inequality.
\end{proof}

\section{Auxiliary}
\begin{lemma}\label{reweight}Let $\F,H$ be two symmetric positive definite matrices in $\R^{d\times d}$ such that 
$H\succeq \lambda\F$. Also, let $u\in\R^d$ and $T$ be a symmetric $k$ tensor for $k\geq 1$. Then
\beq\label{uH1}
|u|_H\geq \lambda^{1/2}|u|_\F,\quad \|T\|_{H}\leq \lambda^{-k/2}\|T\|_\F.
\eeq
\end{lemma}
\begin{proof}
The first bound is by definition. For the second bound, we use the first bound to get
\beq
\|T\|_{H}=\sup_{u\neq 0}\frac{\la T, u^{\otimes k}\ra}{(u^\tr Hu)^{k/2}}\leq \lambda^{-k/2}\sup_{u\neq 0}\frac{\la T, u^{\otimes k}\ra}{(u^\tr \F u)^{k/2}}=\lambda^{-k/2}\|T\|_\F.
\eeq
\end{proof}

\begin{lemma}\label{lma:TVgauss}
Suppose $\Sigma_2^{-1}\succeq \tau\Sigma_1^{-1}$ and $\|\Sigma_2^{-1}-\Sigma_1^{-1}\|_{\Sigma_1^{-1}}\leq\epsilon$. Then
\beq\label{TVD12}
\TV\l(\mathcal N(\mu, \Sigma_1),\;\mathcal N(\mu, \Sigma_2)\r)\leq 2\frac\epsilon\tau\sqrt d.\eeq
\end{lemma}
\begin{proof}
We use Theorem 1.1 of~\cite{devroye2018total} and equality (2) of that work to get that
\beq\label{TV1gauss}
\TV(\mathcal N(\mu,\Sigma_1),\mathcal N(\mu,\Sigma_2)) \leq 2\|\Sigma_1^{-1/2}\Sigma_2\Sigma_1^{-1/2}-I_d\|_F \leq 2\sqrt d\|\Sigma_1^{-1/2}\Sigma_2\Sigma_1^{-1/2}-I_d\|.
\eeq
Next, let $A=\Sigma_1^{1/2}\Sigma_2^{-1}\Sigma_1^{1/2}$ and note that $\lambda_{\min}(A)\geq\tau$ by the assumption $\Sigma_2^{-1}\succeq \tau\Sigma_1^{-1}$. We then have
\beqsn
\|A^{-1}-I_d\| \leq \|A^{-1}\|\|A-I_d\| &= \frac{\|\Sigma_1^{1/2}\Sigma_2^{-1}\Sigma_1^{1/2}-I_d\| }{\lambda_{\min}(\Sigma_1^{1/2}\Sigma_2^{-1}\Sigma_1^{1/2})}=\frac{\|\Sigma_2^{-1}-\Sigma_1^{-1}\|_{\Sigma_1^{-1}}}{\lambda_{\min}(\Sigma_1^{1/2}\Sigma_2^{-1}\Sigma_1^{1/2})}\leq\frac\epsilon\tau.\eeqsn Substituting this bound into~\eqref{TV1gauss} concludes the proof.
 \end{proof}
\begin{lemma}[Lemma E.3 of~\cite{katskew} with $p=0$]\label{aux:gamma} If $ab>1$, then
\beqs
\frac{1}{(2\pi)^{d/2}}\int_{|x|\geq a\sqrt d} e^{-b\sqrt d|x|}dx\leq d\e\l(\l[\frac32+\log a-ab\r]d\r).\eeqs
\end{lemma}

\begin{lemma}[Modified from Lemma 1.3 in Chapter XIV of~\cite{langanalysis}]\label{lma:lang} Let $U\subset\R^d$ be open, and let $\f:U\to \R^d$ be of class $C^1$. Assume that $\f(0)=0$ and $\nabla \f(0)=I_d$. Let $r>0$ be such that $\bar B_r(0)\subset U$. If 
$$|\nabla \f(x)-\nabla \f(0)|\leq \frac14,\qquad\forall |x|\leq r,$$ then $\f$ maps $\bar B_r(0)$ bijectively onto $\bar B_{2r}(0)$.\end{lemma}
Now let $h:\R^d\to\R$ be of class $C^2$, such that $H=\nabla^2h(y_0)$ is symmetric positive definite. Define
$$\tilde h(x)=H^{-1/2}\l(\nabla h\l(y_0+H^{-1/2}x\r)-\nabla h(y_0)\r).$$ Then $\tilde h(0)=0$ and $\nabla \tilde h(0)=I_d$, and one can check that $|\nabla \tilde h(x)-\nabla \tilde h(0)|\leq 1/4$ for all $x\in\bar B_r(0)$ if and only if
\beq\label{Dg}\|\nabla^2h(y) - \nabla^2h(y_0)\|_{H}\leq 1/4,\quad\forall |y-y_0|_{H}\leq r.\eeq
We conclude that if~\eqref{Dg} is satisfied, then $\nabla\tilde h$ maps $\bar B_r(0)$ bijectively onto $\bar B_{2r}(0)$ and hence
$\nabla h$ maps $y_0+H^{-1/2}\bar B_{r}(0)$ bijectively onto $\nabla h(y_0)+H^{1/2}\bar B_{2r}(0)$. We have proved the following corollary.
\begin{corollary}\label{corr:lang}
Let $h:\R^d\to\R$ be $C^2$ such that $H=\nabla^2h(y_0)$ is symmetric positive definite. Suppose~\eqref{Dg} holds, and suppose $z$ is such that 
$$\|z-\nabla h(y_0)\|_H=|H^{-1/2}(z-\nabla h(y_0))|\leq 2r.$$ Then there exists $y$ such that $|y-y_0|_{H}\leq r$ and $\nabla h(y)=z$.
\end{corollary}
Next, it will be convenient to use a weighting matrix $\F$ that does not exactly equal $H=\nabla^2h(y_0)$. We will also only need the case $z=0$.
\begin{corollary}\label{corr:lang2}
Let $h:\R^d\to\R$ be $C^2$ such that $H=\nabla^2h(y_0)$ is symmetric positive definite, and let $\F$ be a symmetric positive definite matrix such that $H\succeq\lambda \F$. Suppose
\beqsn
\|\nabla^2h(y) - \nabla^2h(y_0)\|_{\F}&\leq \frac\lambda4\quad\forall |y-y_0|_{\F}\leq q,\\
|\F^{-1/2}\nabla h(y_0)|&\leq 2\lambda q\eeqsn for some $q\geq0$. Then there exists $y$ such that $|y-y_0|_\F\leq q$ and $\nabla h(y)=0$.
\end{corollary}
\vspace{-15pt}
\begin{proof}
If $|y-y_0|_{H}\leq r=\sqrt\lambda q$, then $|y-y_0|_\F\leq q$ and hence $\|\nabla^2h(y) - \nabla^2h(y_0)\|_{\F}\leq \lambda/4$, which implies $\|\nabla^2h(y) - \nabla^2h(y_0)\|_{H}\leq 1/4$. Also, note that $|H^{-1/2}\nabla h(y_0)|\leq 2\sqrt\lambda q=2r$. Thus the assumptions of Corollary~\ref{corr:lang} are satisfied, so that there exists $y$ such that $|y-y_0|_{H}\leq r$ and $\nabla h(y)=0$. But then $|y-y_0|_\F\leq r/\sqrt\lambda =q$, as desired.
\end{proof}

\bibliographystyle{plain}
\bibliography{bibliogr_Lap}

\end{document}